\documentclass[11pt]{article}
\usepackage[T1]{fontenc}
\usepackage[english]{babel}
\usepackage{microtype}
\usepackage{amsmath,amssymb,amsfonts,amsthm,mathtools}
\usepackage[shortlabels]{enumitem}
\usepackage{geometry}
\usepackage{xcolor}
\usepackage[colorlinks=true,linkcolor=blue,citecolor=blue,urlcolor=blue]{hyperref}
\newtheorem{theorem}{Theorem}[section]
\newtheorem{lemma}[theorem]{Lemma}

\newtheorem{proposition}[theorem]{Proposition}
\newtheorem{corollary}[theorem]{Corollary}
\newtheorem{definition}[theorem]{Definition}

\newtheorem{problem}[theorem]{Problem}
\newtheorem{conj}{Conjecture}

\newcommand{\Aut}{\operatorname{Aut}}
\newcommand{\Emb}{\operatorname{Emb}}

\title{
Subhypergraph counts and proofs of two conjectures of Alon
}

\author{
Peiru Kuang\thanks{Email: \href{mailto:peiru\_k@sjtu.edu.cn}{peiru\_k@sjtu.edu.cn}},\quad
Shuang Sun\thanks{Email: \href{mailto:chocolatesun@sjtu.edu.cn}{chocolatesun@sjtu.edu.cn}},\quad
Yan Wang\thanks{Email: \href{mailto:yan.w@sjtu.edu.cn}{yan.w@sjtu.edu.cn}},\quad
and Jiasheng Zeng\thanks{Email: \href{mailto:jzengbl@hkust.ust.hk}{jzengbl@hkust.ust.hk}}
}
\date{\today}

\begin{document}
\maketitle

\begin{abstract}
 For a fixed integer $r\ge2$ and a finite simple $r$-uniform hypergraph $H$ without isolated vertices, let $N_r(m,H)$ be the maximum number of unlabelled copies of $H$ among all finite simple $r$-uniform hypergraphs with $m$ edges. Friedgut and Kahn~[Isr. J. Math., 1998] proved that $N_r(m,H)=\Theta_H(m^{\rho^*(H)})$, where $\rho^*(H)$ is the fractional edge-cover number of $H$. We sharpen this to $N_r(m,H)=(\lambda_r(H)+o(1))m^{\rho^*(H)}$ as $m\to\infty$, with $\lambda_r(H)>0$, and characterize the limiting $\lambda_r(H)$ by a variational problem over finite hypergraphs. We develop a globally rainbow compression argument combining color-coding with successive vertex cloning. In particular, when $r=2$, the asymptotic formula confirms a conjecture of Alon~[Isr. J. Math., Conjecture~1, 1986]. Moreover, for graphs we express the limiting constant $\lambda_2(H)$ as a supremum over finite cores. We also prove that, for every prescribed number of edges, an extremal graph for a star forest can be chosen to be a star forest, confirming another conjecture of Alon~[Isr. J. Math., Conjecture~2, 1986]. 
\end{abstract}

\section{Introduction}\label{sec:introduction}

Fix an integer $r\ge2$. An $r$-graph is a finite simple $r$-uniform hypergraph, and $e(G)=|E(G)|$ denotes its number of edges. For an $r$-graph $H$ without isolated vertices and an $r$-graph $G$, let $N_r(G,H)$ be the number of unlabelled subhypergraphs of $G$ isomorphic to $H$ (not necessarily induced). For a given nonnegative integer $m$, define
\[
N_r(m,H):=\max_{e(G)=m}N_r(G,H),
\]
where the maximum is taken over all finite simple $r$-graphs $G$ with $m$ edges. The assumption that $H$ has no isolated vertices is needed, since isolated vertices in host graph may be added without changing the number of edges. Unless another limiting parameter is specified, all asymptotic statements fix $r$ and $H$ and let $m$ tend to infinity. For graphs, write $N(G,H)=N_2(G,H)$ and $N(m,H)=N_2(m,H)$ for simplicity.

Alon~\cite{Alon1981} determined the order of magnitude of $N(m,H)$ in terms of neighborhood deficiency. Given a graph $J$ and a vertex $v\in V(J)$, let $N_J(v)=\{u\in V(J):uv\in E(J)\}$ be the neighborhood of $v$, and let $N_J(S)=\bigcup_{v\in S}N_J(v)$ for $S\subseteq V(J)$. For a graph $H$ without isolated vertices, define
\[
D(H):=\max_{S\subseteq V(H)}\bigl(|S|-|N_H(S)|\bigr),
\qquad
\gamma(H):=\frac{|V(H)|+D(H)}2.
\]

\begin{theorem}[Alon~{\cite{Alon1981}}]
Let $H$ be a fixed finite simple graph without isolated vertices. Then $N(m,H)=\Theta_H(m^{\gamma(H)})$ as $m\to\infty$.
\end{theorem}

When $D(H)=0$, Alon~\cite[Theorem~4]{Alon1981} also determined the leading constant. For a general graph, the order-of-magnitude estimate bounds the normalized count between two positive constants but does not imply convergence. Alon subsequently posed the following conjecture.

\begin{conj}[Alon~{\cite[Conjecture~1]{Alon1986}}]\label{conj:alon-first}
For every fixed finite simple graph $H$ without isolated vertices, there exists a constant $b(H)>0$ such that $\lim_{m\to\infty}N(m,H)/m^{\gamma(H)}=b(H)$.
\end{conj}

Friedgut and Kahn~\cite{FriedgutKahn1998} extended Alon's order-of-magnitude result to hypergraphs. To describe the exponent in their theorem, we consider the \emph{fractional independence number} of a graph $H$ without isolated vertices, defined as
\[
\alpha^*(H)=\max\left\{\sum_{v\in V(H)}x_v:x_v\ge0\text{ and }x_u+x_v\le1\text{ for every }uv\in E(H)\right\}.
\]
A vector satisfying these constraints is called a \emph{fractional independent set} of $H$.
One has $\alpha^*(H)=\gamma(H)$; see Proposition~\ref{prop:graph-exponent}. 

For an $r$-graph $H$ without isolated vertices, the analogous program maximizes $\sum_{v\in V(H)}x_v$ over nonnegative vertex weights satisfying $\sum_{v\in e}x_v\le1$ for every $e\in E(H)$. Its dual assigns nonnegative weights to the edges so that the total weight incident with every vertex is at least one. Such an assignment is a fractional edge cover, and the minimum total weight among all such covers is the \emph{fractional edge-cover number}
\[
\rho^*(H)=\min\left\{\sum_{e\in E(H)}\lambda_e:\lambda_e\ge0\text{ and }\sum_{e\ni v}\lambda_e\ge1\text{ for every }v\in V(H)\right\}.
\]
By linear programming duality (see, e.g., Schrijver~\cite{Schrijver1986}), these two optimal values are equal. In particular, when $H$ is a graph, $\rho^*(H)=\alpha^*(H)=\gamma(H)$. Friedgut and Kahn~\cite{FriedgutKahn1998} proved that
\[
N_r(m,H)=\Theta_H(m^{\rho^*(H)})
\]
for every fixed $r$-graph $H$ without isolated vertices. Their result determines the exponent but whether $N_r(m,H)/m^{\rho^*(H)}$ converges in general still remains unknown.

Our first main theorem extends the Friedgut--Kahn order-of-magnitude estimate by proving that the normalized extremal count converges for every fixed $r$-graph without isolated vertices. In particular, it confirms Alon's conjecture when $r=2$. We also establish a bounded-template reduction and characterize the limiting constant by a variational problem over finite hypergraphs. To state this characterization, we first introduce the necessary definitions.

The variational problem counts labelled embeddings. For $r$-graphs $H$ and $G$, let $M_r(G,H)$ be the number of injective maps $\phi:V(H)\to V(G)$ such that $\phi(e)\in E(G)$ for every $e\in E(H)$, and let $M_r(m,H)=\max_{e(G)=m}M_r(G,H)$. Write $\Aut(H)$ for the automorphism group of $H$. Every unlabelled copy corresponds to exactly $|\Aut(H)|$ embeddings, so $M_r(G,H)=|\Aut(H)|N_r(G,H)$ and $M_r(m,H)=|\Aut(H)|N_r(m,H)$. In either extremal problem, the constraint $e(G)=m$ can be relaxed to $e(G)\le m$, since adding pairwise disjoint edges on new vertices brings the edge count to $m$ while preserving all existing copies.

Following the definition of a template in Balister, Bollob\'as and Janson~\cite{BalisterBollobasJanson2025}, we call a finite simple $r$-graph $T$ a \emph{template} and write $I=V(T)$. Given pairwise disjoint finite sets $V_i$ for $i\in I$, the \emph{complete blow-up} of $T$ with these vertex classes is the $r$-graph $G$ with vertex set $V(G)=\bigcup_{i\in I}V_i$. An $r$-element set $\{x_1,\ldots,x_r\}\subseteq V(G)$ is an edge of $G$ if and only if there exist distinct $i_1,\ldots,i_r\in I$ such that $x_j\in V_{i_j}$ for every $1\le j\le r$ and $\{i_1,\ldots,i_r\}\in E(T)$. For $r=2$, our complete blow-ups of $T$ are precisely the finite induced subgraphs, up to isomorphism, of their infinite blow-up of $T$ with every template vertex marked \emph{empty}. A \emph{homomorphism} from $H$ to $T$ is a map $\theta:V(H)\to V(T)$ such that $\theta(e)\in E(T)$ for every $e\in E(H)$. We write
\[
\operatorname{Hom}(H,T)=\left\{\theta:V(H)\to V(T):\theta(e)\in E(T)\text{ for every }e\in E(H)\right\}.
\]
For $i\in I$, let $k_i(\theta)=|\theta^{-1}(i)|$. For integers $a,k\ge0$, write $(a)_0=1$ and $(a)_k=a(a-1)\cdots(a-k+1)$ when $k\ge1$. If a class has size $a$, the factor $(a)_{k_i(\theta)}$ counts the distinct images of all vertices assigned to that class, including vertices in different components of $H$.

The vertex classes may have different growth rates as the number of edges increases. To describe these rates and the associated leading coefficients, consider a triple $D=(\alpha,b,s)$ on $T$. The exponent vector $\alpha=(\alpha_i)_{i\in I}$ has nonnegative entries and satisfies
\[
\alpha(f):=\sum_{i\in f}\alpha_i\le1\qquad\text{for every }f\in E(T).
\]
Let $Z(\alpha):=\{i\in I:\alpha_i=0\}$. The components $b$ and $s$ assign an integer $b_i\ge1$ to every $i\in Z(\alpha)$ and a real number $s_i>0$ to every $i\in I\setminus Z(\alpha)$; set $w_i=b_i$ in the first case and $w_i=s_i$ in the second. Call an edge $f$ \emph{critical} if $\alpha(f)=1$ and \emph{subcritical} if $\alpha(f)<1$. We call $D$ \emph{feasible} if
\[
\sum_{\substack{f\in E(T)\\\alpha(f)=1}}\prod_{i\in f}w_i\le1.
\]
For fixed $D=(\alpha,b,s)$, the class indexed by $i$ has size $b_i$ when $\alpha_i=0$ and size $\lfloor s_i m^{\alpha_i}\rfloor$ when $\alpha_i>0$. Critical edges contribute at most $m$ edges in total, while subcritical edges contribute $o(m)$. The realization argument in Section~\ref{sec:preliminaries} adjusts the growing classes to accommodate these lower-order contributions without exceeding $m$ edges.

For each homomorphism $\theta$, the vertex weights $x_v=\alpha_{\theta(v)}$ satisfy the dual constraints defining $\rho^*(H)$, and consequently $\sum_i k_i(\theta)\alpha_i\le\rho^*(H)$. We call a homomorphism leading with respect to $\alpha$ when equality holds. Its contribution to the normalized embedding count is a product of falling factorials for the fixed-size classes and powers of the leading coefficients for the growing classes. Summing these contributions, we obtain
\begin{equation}\label{eq:graded-objective}
\Phi_{H,T}(D)=\sum_{\substack{\theta\in\operatorname{Hom}(H,T)\\\sum_{i\in I}k_i(\theta)\alpha_i=\rho^*(H)}}\left(\prod_{i\in Z(\alpha)}(b_i)_{k_i(\theta)}\right)\left(\prod_{i\notin Z(\alpha)}s_i^{k_i(\theta)}\right).
\end{equation}
Define the template constant and the supremum over templates by
\[
c_H(T)=\sup_D\Phi_{H,T}(D)
\qquad\text{and}\qquad
\Lambda_r(H)=\sup_Tc_H(T),
\]
where the first supremum is taken over feasible triples $D$ on $T$ and the second over finite simple $r$-graphs $T$. These definitions depend only on the finite templates and their parameters, independently of the extremal function $M_r(m,H)$.

\begin{theorem}\label{thm:hypergraph-main}
Let $r\ge2$, and let $H$ be a fixed finite simple $r$-uniform hypergraph without isolated vertices. Set $\rho=\rho^*(H)$. Then
\[
0<\Lambda_r(H)\le(r!)^\rho
\]
and
\[
\lim_{m\to\infty}\frac{M_r(m,H)}{m^\rho}=\Lambda_r(H).
\]
Consequently,
\[
\lim_{m\to\infty}\frac{N_r(m,H)}{m^\rho}=\lambda_r(H),
\qquad
\lambda_r(H)=\frac{\Lambda_r(H)}{|\Aut(H)|}.
\]
Moreover, for every $\varepsilon>0$, there is an integer $L=L(H,r,\varepsilon)$ such that, for every finite simple $r$-graph $G$, there exists a complete blow-up $G'$ of a template with at most $L$ vertices satisfying $e(G')\le e(G)$ and $M_r(G',H)\ge(1-\varepsilon)M_r(G,H)$.

\end{theorem}

When $r=2$, we have $\rho^*(H)=\gamma(H)$, so Theorem~\ref{thm:hypergraph-main} proves Conjecture~\ref{conj:alon-first}. For graphs, we obtain a more explicit formula for the leading constant by separating a finite core, vertex classes of order $m$ specified by their neighborhoods in the core, and a complete graph of order $\sqrt m$. We next define the contributions to this formula.

A nonempty graph $B$ without isolated vertices is called \emph{balanced} if $\alpha^*(B)=|V(B)|/2$. For $U\subseteq V(H)$, write $H[U]$ for the subgraph induced by $U$. For $S\subseteq V(H)$, put $R_S=V(H)\setminus S$, let $Z_S$ be the set of isolated vertices of $H[R_S]$, and let $B_S=H[R_S\setminus Z_S]$. Define $\beta(S)=|Z_S|+\alpha^*(B_S)$, and call $S$ \emph{admissible} if $\beta(S)=\alpha^*(H)$ and $B_S$ is empty or balanced. The set $S$ will be embedded into the core, the vertices in $Z_S$ will be chosen from the classes of linear size, and $B_S$ will be embedded into the complete graph of order $\sqrt m$.

Let $F$ be a finite graph, possibly empty, which we call the \emph{core}. For each nonempty $A\subseteq V(F)$, let $y_A\ge0$ denote the normalized size (scaled by $m$) of the class whose neighborhood in $F$ is $A$. To ensure that the edges between these classes and the core do not exceed $m$ at leading order, we require
\[
s(y):=\sum_{\varnothing\ne A\subseteq V(F)}|A|y_A\le1.
\]
When $F$ is empty, $y$ is the unique empty vector and $s(y)=0$. For $\Gamma\subseteq V(F)$, let $Y_\Gamma=\sum_{A\supseteq\Gamma,\,A\ne\varnothing}y_A$. Write $\psi:H[S]\hookrightarrow F$ for an injective edge-preserving map, and define $\Gamma_\psi(z)=\psi(N_H(z)\cap S)$ for $z\in Z_S$. The set $\Gamma_\psi(z)$ is nonempty because $H$ has no isolated vertices and $z$ has no neighbor in $R_S$. 
Thus, $Y_{\Gamma_\psi(z)}$ is the normalized number of choices for $z$. The remaining fraction $1-s(y)$ of edges supports a clique on roughly $\sqrt{2(1-s(y))m}$ vertices joined completely to $F$, with the $o(m)$ extra edges absorbed by slightly shrinking the classes.

Summing over admissible sets and their embeddings into the core, we define
\[
\Phi_H^{\mathrm{core}}(F,y):=\sum_{\substack{S\subseteq V(H)\\S\text{ admissible}}}\ \sum_{\psi:H[S]\hookrightarrow F}\left(\prod_{z\in Z_S}Y_{\Gamma_\psi(z)}\right)[2(1-s(y))]^{|V(B_S)|/2}.
\]
Here, empty products are taken to be $1$, and the factor $[2(1-s(y))]^{|V(B_S)|/2}$ is set to $1$ whenever $B_S$ is empty (even if $s(y)=1$). Define
\[
\Lambda_{\mathrm{core}}(H):=\sup_{F,y}\Phi_H^{\mathrm{core}}(F,y),
\]
where the supremum is taken over finite graphs $F$ and nonnegative vectors $y$ satisfying $s(y)\le1$. For a fixed $F$, these constraints imply $0\le y_A\le1/|A|$, so the feasible set is compact and the continuous objective attains its maximum. Let $K_t$ denote the complete graph on $t$ vertices, including the empty graph $K_0$, and write $\lambda(H)=\lambda_2(H)$. This formula identifies Alon's constant and shows that complete cores suffice.

\begin{corollary}\label{cor:alon-first}
Let $H$ be a finite simple graph without isolated vertices. Then $\gamma(H)=\rho^*(H)$ and
\[
\lim_{m\to\infty}\frac{N(m,H)}{m^{\gamma(H)}}=\lambda(H),
\]
where
\[
\lambda(H)=\frac{\Lambda_2(H)}{|\Aut(H)|}=\frac{\Lambda_{\mathrm{core}}(H)}{|\Aut(H)|}.
\]
Moreover,
\[
\Lambda_{\mathrm{core}}(H)=\sup_{t\ge0}\ \max_y\Phi_H^{\mathrm{core}}(K_t,y),
\]
where $t$ ranges over nonnegative integers and the maximum is over feasible weights on $K_t$.
\end{corollary}

The maximum over weights is therefore attained for each fixed core. Whether a single finite core attains the supremum over all core orders remains open; see Problem~\ref{prob:finite-core-order}.

Another conjecture of Alon concerns the exact extremal graphs for each prescribed number of edges. A star $K_{1,a}$ has one center joined to $a\ge1$ leaves, and a star forest is a vertex-disjoint union of stars; the empty union is permitted. Alon~\cite[Theorem~2(ii)]{Alon1986} already established the existence of the normalized limit for star forests. He also asked whether the extremal count for a star forest can always be attained by a star forest.

\begin{conj}[Alon~{\cite[Conjecture~2]{Alon1986}}]\label{conj:alon-second}
Let $H$ be a star forest. Then, for every integer $m>0$, or at least for all sufficiently large integers $m$, there exists a star forest $G_m$ such that $e(G_m)=m$ and $N(G_m,H)=N(m,H)$.
\end{conj}

F\"uredi~\cite{Furedi1992} proved Conjecture~\ref{conj:alon-second} for all sufficiently large $m$ when $H$ is a disjoint union of $t\ge2$ stars, each having more than $\log_2(t+1)$ edges. Our second main result proves Conjecture~\ref{conj:alon-second} for every star forest $H$ and every $m\ge0$.

\begin{theorem}\label{thm:alon-second}
Let $H$ be a disjoint union of stars. Then, for every integer $m\ge0$, there exists a disjoint union of stars $S_m$ such that
$
e(S_m)=m
$
and
$
N(S_m,H)=N(m,H).
$
\end{theorem}

For paths, we can compute the leading constants explicitly from the finite-core formula. The problem of counting paths in graphs with a given number of edges was previously studied by Bollob\'as and Sarkar~\cite{BollobasSarkar2001,BollobasSarkar2003}; in Section~\ref{sec:concluding-remarks}, we determine these constants and discuss questions concerning attainment and quantitative bounds.

Additional restrictions on the host graph lead to related extremal problems. Gerbner, Nagy, Patk\'os and Vizer~\cite{GerbnerNagyPatkosVizer2021} prescribe both the order and the size of the host. Wang, Xu, Zeng and Zhang~\cite{wang2025generalized} study clique counts in host graph with a given number of edges and a forbidden subgraph. For all the problems considered in this paper, the only restriction on the host graph is the prescribed number of edges.

To bound embedding counts, we use a weighted form of the projection inequality of Chung, Graham, Frankl and Shearer~\cite{ChungGrahamFranklShearer1986}, which combines bounds on the restrictions of embeddings to individual edges. This inequality can be proved using Shannon entropy. For a finite random variable $X$ with probabilities $p_x=\mathbb P(X=x)$, its \emph{Shannon entropy} is $\mathrm H(X)=-\sum_xp_x\log p_x$, where $0\log0=0$ and logarithms are natural. If $X$ is uniform on a nonempty finite set $\mathcal F$, then $\mathrm H(X)=\log|\mathcal F|$. Radhakrishnan~\cite{Radhakrishnan1997} and Cutler and Radcliffe~\cite{CutlerRadcliffe2011} used entropy to prove upper bounds for perfect matchings in bipartite and general graphs, respectively. Ilinca and Kahn~\cite{IlincaKahn2013} obtained upper bounds for matchings of a given size, while Linial and Luria~\cite{LinialLuria2014} established asymptotic upper bounds for high-dimensional permutations. Kahn~\cite{Kahn2001} and Sason~\cite{Sason2021} applied entropy to independent-set counts in regular and irregular bipartite graphs, respectively. Conlon, Kim, Lee and Lee~\cite{ConlonKimLeeLee2018} obtained lower bounds on homomorphism counts for a class of bipartite graphs. Chao and Yu~\cite{ChaoYu2026EntropyTuran} proved Tur\'an's theorem using entropy and determined Tur\'an densities for a family of uniform hypergraphs.

\paragraph{Organization}
The remainder of the paper is organized as follows. Section~\ref{sec:preliminaries} establishes the fractional-cover estimates and the exact counting and realization formulas for finite templates. Section~\ref{sec:hypergraph-case} proves the subhypergraph counting theorem (Theorem~\ref{thm:hypergraph-main}). Section~\ref{sec:graph-case} derives the finite-core formula for the leading constant in the graph case (Corollary~\ref{cor:alon-first}) and proves Alon's conjecture for star forests (Theorem~\ref{thm:alon-second}). In Section~\ref{sec:concluding-remarks}, we determine the path constants and discuss template order, attainment, equality in the projection bound and convergence rates.

\section{Preliminaries}\label{sec:preliminaries}

\subsection{Fractional Covers and Projection Inequalities}

Fix an integer $r\ge2$. For $r$-graphs $H$ and $G$, let
$\Emb_r(H,G)$ denote the set of embeddings of $H$ into $G$,
so that $M_r(G,H)=|\Emb_r(H,G)|$. If $V(H)=\varnothing$, then $\Emb_r(H,G)$ consists of the unique empty embedding, and hence $M_r(G,H)=1$.

Throughout this paper, $H$ has no isolated vertices. Otherwise, for $m\ge e(H)$, adjoining arbitrarily many isolated host vertices makes the embedding count unbounded. If $H$ is nonempty, deleting isolated vertices from any $m$-edge host preserves its embedding count and leaves at most $rm$ vertices. Since there are only finitely many such hosts up to isomorphism, the maximum defining $M_r(m,H)$ is attained.


Fractional edge covers allow us to combine edgewise bounds into a global bound on the number of embeddings.

\begin{definition}\label{def:fractional-edge-cover}
Let $\mathcal C(H)$ be the set of fractional edge covers of $H$, as defined in Section~\ref{sec:introduction}. Define
\[
\mathcal P(H)=
\left\{
x\in\mathbb R_{\ge0}^{V(H)}:
\sum_{v\in e}x_v\le1
\text{ for every }e\in E(H)
\right\}.
\]
We call the elements of $\mathcal P(H)$ dual-feasible. If $V(H)=\varnothing$, both sets consist of the empty vector and $\rho^*(H)=0$.
\end{definition}

We adopt the convention $0^0=1$. Unless otherwise stated, we assume throughout the remainder of this section that $H$ is nonempty.

Fractional edge covers bound embedding counts, while the dual variables describe the growth rates of blow-up classes. The edge weights $\lambda_e=1$ and the vertex weights $x_v=0$ show that the fractional edge-cover program and its dual are feasible. Since $H$ has no isolated vertices, every dual-feasible coordinate satisfies $x_v\le1$. By linear-programming duality (see Schrijver~\cite{Schrijver1986}), both optima are attained and equal $\rho^*(H)$. Summing the covering constraints over all vertices, we obtain $|V(H)|\le r\sum_e\lambda_e$, and hence $\rho^*(H)\ge|V(H)|/r>0$.

We use the following product theorem for families of sets.

\begin{lemma}[Chung, Graham, Frankl and Shearer~{\cite{ChungGrahamFranklShearer1986}}]\label{thm:product-traces}
Let $U$ be a finite set and let $q\ge1$ be an integer. Let $B_1,\ldots,B_t$ be subsets of $U$ such that every element of $U$ belongs to at least $q$ of them, counting repetitions. For a nonempty family $\mathcal A\subseteq 2^U$, put $\mathcal A_i=\{A\cap B_i:A\in\mathcal A\}$ for $1\le i\le t$. Then $|\mathcal A|^q\le\prod_{i=1}^t|\mathcal A_i|$.
\end{lemma}

A fractional edge cover becomes effective for counting through a weighted projection inequality. We state the inequality with arbitrary nonnegative real weights because those are the weights provided by the linear program. Let $W$ be a finite coordinate set, let $\mathcal X_w$ be a finite set for every $w\in W$, and let $\mathcal F$ be a family of assignments in $\prod_{w\in W}\mathcal X_w$. For $A\subseteq W$, denote by $\pi_A(\mathcal F)$ the set of restrictions to $A$ of assignments in $\mathcal F$.

\begin{lemma}\label{lem:weighted-projection}
Let $\mathcal F$ be a nonempty finite family of assignments on $W$. Let $(A_j)_{j\in J}$ be a finite family of subsets of $W$, and let $\lambda_j\ge0$ for $j\in J$. If $\sum_{j:w\in A_j}\lambda_j\ge1$ for every $w\in W$, then $|\mathcal F|\le\prod_{j\in J}|\pi_{A_j}(\mathcal F)|^{\lambda_j}$.
\end{lemma}

\begin{proof}
If $W=\varnothing$, then $\mathcal F$ and each $\pi_{A_j}(\mathcal F)$ consist of the empty assignment. Thus $|\mathcal F|=1=\prod_{j\in J}1^{\lambda_j}=\prod_{j\in J}|\pi_{A_j}(\mathcal F)|^{\lambda_j}$. Otherwise, identify each assignment $x\in\mathcal F$ with the set $\{(w,x_w):w\in W\}$ in $U=\bigcup_{w\in W}(\{w\}\times\mathcal X_w)$. For each $j\in J$, put $B_j=\bigcup_{w\in A_j}(\{w\}\times\mathcal X_w)$. Intersections of the encoded sets with $B_j$ correspond bijectively to $\pi_{A_j}(\mathcal F)$. If $\lambda_j=a_j/q$ for nonnegative integers $a_j$ and a common positive denominator $q$, repeat $B_j$ exactly $a_j$ times. Every element of $U$ is covered at least $q$ times. By Lemma~\ref{thm:product-traces}, we have $|\mathcal F|^q\le\prod_j|\pi_{A_j}(\mathcal F)|^{a_j}$. Taking $q$th roots proves the rational case. For arbitrary nonnegative real weights, approximate each $\lambda_j$ from above by rational numbers. The covering constraints remain valid, and the result follows by continuity, since every projection of the nonempty family $\mathcal F$ has positive cardinality.
\end{proof}

For embeddings of $H$ into an $m$-edge host, projection onto an edge of $H$ produces an ordered host edge and hence has at most $r!m$ possible values. Applying the weighted projection inequality with an optimal fractional edge cover, we obtain a uniform bound with exponent $\rho^*(H)$.

\begin{proposition}\label{prop:global-projection-bound}
If $G$ is an $r$-graph with $m$ edges, then $M_r(G,H)\le(r!m)^{\rho^*(H)}$.
\end{proposition}

\begin{proof}
If $\Emb_r(H,G)$ is empty, the bound holds. Otherwise, let $\mathcal F=\Emb_r(H,G)$ and choose an optimal fractional edge cover $\lambda$ of $H$. For every $e\in E(H)$, the projection $\pi_e(\mathcal F)$ has at most $r!m$ values, since an embedding maps $e$ bijectively onto one of the $m$ host edges. By Lemma~\ref{lem:weighted-projection}, we have
\[
M_r(G,H)\le\prod_{e\in E(H)}|\pi_e(\mathcal F)|^{\lambda_e}\le(r!m)^{\sum_e\lambda_e}=(r!m)^{\rho^*(H)}.\qedhere
\]
\end{proof}

Taking the maximum over all $m$-edge hosts, we obtain $M_r(m,H)\le(r!m)^{\rho^*(H)}$. 
The projection bound determines the correct upper exponent but does not determine the leading constant. Complete blow-ups provide finite-dimensional families of hosts in which the edge constraint and the embedding count can both be computed exactly.

\subsection{Finite Templates and Exact Counting}

Let $T$ be a finite simple $r$-graph with vertex set $I$, and let $\mathbf n=(n_i)_{i\in I}\in\mathbb Z_{\ge0}^{I}$. Choose pairwise disjoint sets $V_i$ with $|V_i|=n_i$, and write $T(\mathbf n)$ for the complete blow-up of $T$ with these vertex classes, as defined in Section~\ref{sec:introduction}. Define $\pi:\bigcup_{i\in I}V_i\to I$ by $\pi(x)=i$ for $x\in V_i$.

Define the edge polynomial of $T$ by
\[
Q_T(\mathbf n):=\sum_{f\in E(T)}\prod_{i\in f}n_i.
\]
The summand indexed by $f$ counts the edges of the blow-up whose set of class indices is $f$.

Using $\operatorname{Hom}(H,T)$, $k_i(\theta)$ and the falling factorials from Section~\ref{sec:introduction}, define the labelled embedding polynomial by
\[
P_{H,T}(\mathbf n):=\sum_{\theta\in\operatorname{Hom}(H,T)}\prod_{i\in I}(n_i)_{k_i(\theta)}.
\]
Although this expression is written in the falling-factorial basis, it is a polynomial in the variables $(n_i)_{i\in I}$.

\begin{lemma}\label{lem:exact-template-count}
For every $\mathbf n\in\mathbb Z_{\ge0}^{I}$, the complete blow-up satisfies $e(T(\mathbf n))=Q_T(\mathbf n)$ and $M_r(T(\mathbf n),H)=P_{H,T}(\mathbf n)$.
\end{lemma}

\begin{proof}
For every $f\in E(T)$, exactly $\prod_{i\in f}n_i$ edges of $T(\mathbf n)$ have class indices $f$. These families are disjoint for distinct template edges, so the sum of their sizes is $Q_T(\mathbf n)$.

Every embedding $\phi$ determines a unique homomorphism $\theta=\pi\circ\phi$. For a fixed homomorphism $\theta$, embeddings inducing $\theta$ correspond bijectively to choices of injections $\theta^{-1}(i)\to V_i$ for all $i\in I$. The union of these injections is globally injective because the classes are pairwise disjoint, and it preserves every edge because $\theta$ is a homomorphism and $T(\mathbf n)$ is a complete blow-up. There are $\prod_i(n_i)_{k_i(\theta)}$ such choices. Summing over the disjoint families indexed by $\theta\in\operatorname{Hom}(H,T)$ proves $M_r(T(\mathbf n),H)=P_{H,T}(\mathbf n)$.
\end{proof}

The formulas remain valid even when some classes are empty or too small, since $(n_i)_k=0$ when $n_i<k$ and $(n_i)_0=1$. Since $H$ has no isolated vertices, no homomorphism from $H$ to $T$ uses an isolated vertex of $T$, so deleting such vertices leaves both counting polynomials unchanged. If $E(T)=\varnothing$, then $\operatorname{Hom}(H,T)=\varnothing$ and $P_{H,T}(\mathbf n)=0$. 


\subsection{Parameters and Their Realization}

Fix a finite simple $r$-graph $T$ with vertex set $I$.
For a triple $D=(\alpha,b,s)$ on $T$ as in Section~\ref{sec:introduction}, define
\[
B_T(D)=\sum_{\substack{f\in E(T)\\\alpha(f)=1}}\prod_{i\in f}w_i.
\]
Thus $D$ is feasible if and only if $B_T(D)\le1$.
We state the complementary slackness conditions for fractional edge covers in the form needed below; see Schrijver~\cite{Schrijver1986}.

\begin{lemma}\label{lem:cover-slackness}
Let $\lambda$ be an optimal fractional edge cover of $H$ and let $x\in\mathcal P(H)$. Then $\sum_{v\in V(H)}x_v=\rho^*(H)$ if and only if both of the following conditions hold:
\begin{enumerate}[(i)]
\item $\sum_{v\in e}x_v=1$ for every $e\in E(H)$ with $\lambda_e>0$;
\item $\sum_{e\ni v}\lambda_e=1$ for every $v\in V(H)$ with $x_v>0$.
\end{enumerate}
\end{lemma}

For $\theta\in\operatorname{Hom}(H,T)$, let $g_D(\theta)=\sum_{v\in V(H)}\alpha_{\theta(v)}=\sum_{i\in I}k_i(\theta)\alpha_i$. Its dependence on $D$ is only through the exponent vector $\alpha$.
For class sizes with the growth rates prescribed by $D$, the $\theta$-term of $P_{H,T}(\mathbf n)$ is $O_{H,T,D}(m^{g_D(\theta)})$. The dual program for $\rho^*(H)$ gives $g_D(\theta)\le \rho^*(H)$, with equality characterized by complementary slackness.

\begin{lemma}\label{lem:state-grade}
Every homomorphism $\theta\in\operatorname{Hom}(H,T)$ satisfies $g_D(\theta)\le\rho^*(H)$. Let $\lambda$ be any optimal fractional edge cover of $H$ and let $c_v=\sum_{e\ni v}\lambda_e$. Equality holds if and only if $\lambda_e>0$ implies that $\theta(e)$ is critical and $\alpha_{\theta(v)}>0$ implies that $c_v=1$.
\end{lemma}

\begin{proof}
Set $x_v=\alpha_{\theta(v)}$. Since $H$ and $T$ are $r$-uniform, $\theta$ is injective on each edge. Hence $\sum_{v\in e}x_v=\alpha(\theta(e))\le1$, so $x$ is dual-feasible and $g_D(\theta)\le\rho^*(H)$. By Lemma~\ref{lem:cover-slackness}, equality holds precisely when $\alpha(\theta(e))=1$ for every edge $e$ with $\lambda_e>0$ and $c_v=1$ for every vertex $v$ with $\alpha_{\theta(v)}>0$. These are exactly the stated conditions.
\end{proof}

By Lemma~\ref{lem:state-grade}, only homomorphisms $\theta$ satisfying $g_D(\theta)=\rho^*(H)$ can contribute to the leading term. If $\theta$ is leading and $i=\theta(v)$ for some $v\in V(H)$, then $c_v\ge1$ ensures that some edge $e\ni v$ has $\lambda_e>0$; by Lemma~\ref{lem:state-grade}, the critical edge $\theta(e)$ contains $i$. If nonleading homomorphisms exist, define $\delta_{H,T,\alpha}=\min\{\rho^*(H)-g_D(\theta):\theta\in\operatorname{Hom}(H,T),\ g_D(\theta)<\rho^*(H)\}$. The set over which the minimum is taken is finite and consists of positive numbers, so $\delta_{H,T,\alpha}>0$. If every homomorphism is leading, set $\delta_{H,T,\alpha}=1$; assertions concerning nonleading homomorphisms are then vacuous.

To compare the variational constant $\Lambda_r(H)$ from Section~\ref{sec:introduction} with the extremal counts, we realize each feasible triple by integer class sizes without exceeding the given number of edges. For fixed $T$ and $\alpha$, the exponents of the subcritical edges are bounded above by some $\sigma<1$, and this gap leaves enough room for rounding and for completing the host to exactly $m$ edges.

\begin{lemma}\label{lem:datum-realization}
Let $D$ be a feasible triple on $T$. There is a sequence of complete blow-ups $G_m$ of $T$ such that $e(G_m)\le m$ for all sufficiently large integers $m$ and $\lim_{m\to\infty}M_r(G_m,H)/m^{\rho^*(H)}=\Phi_{H,T}(D)$. By adding edges on new vertices, one obtains a sequence $\widehat G_m$ with exactly $m$ edges and $\liminf_{m\to\infty}M_r(\widehat G_m,H)/m^{\rho^*(H)}\ge\Phi_{H,T}(D)$.
\end{lemma}

\begin{proof}
Let $\sigma$ be the maximum of $0$ and the values $\alpha(f)$ of the subcritical edges, so $\sigma<1$. Let $\eta=(1+\sigma)/2$ and $u_m=m-m^\eta$. For sufficiently large $m$, define $n_i(m)=b_i$ when $\alpha_i=0$ and $n_i(m)=\lfloor s_i u_m^{\alpha_i}\rfloor$ when $\alpha_i>0$, and let $G_m=T(\mathbf n(m))$. The critical edges contribute at most $u_m B_T(D)$ edges, and the subcritical edges contribute $O_{T,D}(m^\sigma)$. By Lemma~\ref{lem:exact-template-count} and feasibility,
\[
e(G_m)\le u_m B_T(D)+O_{T,D}(m^\sigma)\le m-m^\eta+O_{T,D}(m^\sigma)\le m
\]
for all sufficiently large $m$, since $\sigma<\eta$.

For a fixed homomorphism $\theta$ and $i\notin Z(\alpha)$, put $k=k_i(\theta)$. Then $(n_i(m))_k/u_m^{k\alpha_i}\to s_i^k$, since $u_m^{\alpha_i}\to\infty$ and the rounding error is less than $1$. For $i\in Z(\alpha)$, the factor remains exactly $(b_i)_{k_i(\theta)}$. Thus the contribution of every leading homomorphism, after normalization by $m^{\rho^*(H)}$, converges to its summand in \eqref{eq:graded-objective}, using $u_m/m\to1$. Every nonleading homomorphism has $g_D(\theta)\le\rho^*(H)-\delta_{H,T,\alpha}$, so its contribution is $O_{H,T,D}(u_m^{g_D(\theta)})=o(m^{\rho^*(H)})$. Summing over the finitely many homomorphisms and applying Lemma~\ref{lem:exact-template-count} proves the limit for $G_m$.

For the finitely many remaining $m$, set $G_m=T(\mathbf 0)$. By adding $m-e(G_m)$ pairwise disjoint $r$-edges on new vertices, we obtain a simple $r$-graph $\widehat G_m$ with exactly $m$ edges while preserving every embedding in $G_m$. This proves the asserted lower limit, allowing for any additional embeddings created by the new edges.
\end{proof}

The realization lemma converts every feasible objective value into an asymptotic lower bound, while Proposition~\ref{prop:global-projection-bound} bounds the same realized counts uniformly from above. Taking the two suprema, we obtain the basic bounds for the hypergraph variational constant.

\begin{proposition}\label{prop:variational-bounds}
For every nonempty $r$-graph $H$ without isolated vertices, the variational value satisfies $0<\Lambda_r(H)\le(r!)^{\rho^*(H)}$. Moreover, $\liminf_{m\to\infty}M_r(m,H)/m^{\rho^*(H)}\ge\Lambda_r(H)$.
\end{proposition}

\begin{proof}
Fix a template $T$ and a feasible triple $D$. By Lemma~\ref{lem:datum-realization} and Proposition~\ref{prop:global-projection-bound}, we have
\[
\Phi_{H,T}(D)\le\liminf_{m\to\infty}\frac{M_r(m,H)}{m^{\rho^*(H)}}\le(r!)^{\rho^*(H)}.
\]
Taking the supremum over $T$ and $D$ proves both the upper bound for $\Lambda_r(H)$ and the asserted lower-limit inequality.

For positivity, take $T=H$ and choose an optimal dual vector $x$ of $H$ as its exponent vector. Set $b_i=1$ when $x_i=0$ and $s_i=\tau:=1/(e(H)+1)$ when $x_i>0$, and let $D=(x,b,s)$. Every critical edge contains a vertex $i$ with $x_i>0$, so its weight product is at most $\tau$ and $B_T(D)\le e(H)\tau<1$. The identity homomorphism is leading and contributes $\tau^{|\{v:x_v>0\}|}>0$, proving $\Lambda_r(H)>0$.
\end{proof}

The lower bound for the limiting constant follows from Lemma~\ref{lem:datum-realization} and Proposition~\ref{prop:variational-bounds}. To prove the upper bound, we first optimize over the blow-ups of a fixed template and then approximate arbitrary hosts by templates of bounded order.
\section{The Hypergraph Limit}\label{sec:hypergraph-case}

\subsection{The Fixed-Template Limit}

Throughout this section, fix $r\ge2$ and a finite simple $r$-graph $H$. If $V(H)=\varnothing$, then $\rho^*(H)=0$ and every host has exactly one embedding of $H$. The unique empty map contributes $1$ to every template objective, so $\Lambda_r(H)=1$, and the empty blow-up satisfies the final assertion in Theorem~\ref{thm:hypergraph-main}. We may therefore assume that $H$ is nonempty. 
Write $h=|V(H)|$ and $\rho=\rho^*(H)>0$.

In the fixed-template problem, we require $n_i\ge1$ for every $i\in V(T)$. This restriction causes no loss when considering all templates. Indeed, if $J=\{i\in V(T):n_i>0\}$, then $T(\mathbf n)$ is isomorphic to the complete blow-up of $T[J]$ with positive class-size vector $\mathbf n|_J$. However, for a fixed template $T$, allowing zero class sizes would instead give the limiting constant $\max_{J\subseteq V(T)}c_H(T[J])$. This distinction disappears after taking the supremum over all finite templates in the definition of $\Lambda_r(H)$.

Let $T$ be a fixed finite simple $r$-graph with vertex set $I$. For every integer $m\ge e(T)$, define
\begin{equation}\label{eq:fixed-template-extremal}
A_{H,T}(m):=\max\left\{P_{H,T}(\mathbf n):\mathbf n\in\mathbb Z_{>0}^{I}\text{ and }Q_T(\mathbf n)\le m\right\}.
\end{equation}
This maximum exists. No homomorphism from $H$ to $T$ maps a vertex to an isolated vertex of $T$. If a homomorphism sent some $v\in V(H)$ to such a vertex, an edge of $H$ containing $v$ would have no image edge in $T$. Resetting the class size of every isolated template vertex to $1$ consequently leaves both counting polynomials unchanged. If $i$ is not isolated, choose $f\in E(T)$ with $i\in f$. Since the class sizes are positive and satisfy the edge constraint, we have $n_i\le\prod_{j\in f}n_j\le Q_T(\mathbf n)\le m$. After the isolated coordinates have been reset, every coordinate is therefore either fixed at $1$ or belongs to the finite set $\{1,\ldots,m\}$, so the maximum is taken over a finite nonempty set. If $E(T)=\varnothing$, then there is no homomorphism from $H$ to $T$, and both $A_{H,T}(m)$ and $c_H(T)$ are zero.

The lower bound follows from Lemma~\ref{lem:datum-realization}. For the upper bound, we extract limiting exponents from the normalized logarithms of the class sizes. A class may have limiting exponent zero while its size grows subpolynomially, so the corresponding integer parameter must be allowed to vary along the maximizing sequence.

\begin{proposition}\label{prop:fixed-template-limit}
For every fixed finite simple $r$-graph $T$,
\[
\lim_{m\to\infty}\frac{A_{H,T}(m)}{m^\rho}=c_H(T).
\]
\end{proposition}

\begin{proof}
Fix a feasible triple $D$ on $T$. In the construction in Lemma~\ref{lem:datum-realization}, every class with $\alpha_i=0$ has the positive size $b_i$, and every class with $\alpha_i>0$ has size $\lfloor s_i u_m^{\alpha_i}\rfloor$, which is positive for all sufficiently large $m$. The resulting blow-up therefore has all vertex classes nonempty and at most $m$ edges. Using Lemmas~\ref{lem:exact-template-count} and~\ref{lem:datum-realization} together with the definition of $A_{H,T}(m)$, we obtain
\[
\liminf_{m\to\infty}\frac{A_{H,T}(m)}{m^\rho}\ge\Phi_{H,T}(D).
\]
Taking the supremum over all feasible triples $D$, we have $\liminf_m A_{H,T}(m)/m^\rho\ge c_H(T)$.

For the reverse inequality, we first apply Lemma~\ref{lem:exact-template-count} and Proposition~\ref{prop:global-projection-bound} to obtain $A_{H,T}(m)\le(r!m)^\rho$, so the normalized sequence has a finite upper limit. Choose integers $m_\ell\to\infty$ such that $A_{H,T}(m_\ell)/m_\ell^\rho$ converges to this upper limit. For every $\ell$, choose a maximizing vector $\mathbf n^{(\ell)}=(n_i^{(\ell)})_{i\in I}$ in \eqref{eq:fixed-template-extremal}, with $n_i^{(\ell)}=1$ whenever $i$ is isolated in $T$. By the preceding coordinate bound, $1\le n_i^{(\ell)}\le m_\ell$ for every nonisolated $i$. After passage to a subsequence, the normalized logarithms of the class sizes converge for all $i\in I$, and we set
\begin{equation}\label{eq:template-log-grade}
\alpha_i:=\lim_{\ell\to\infty}\frac{\log n_i^{(\ell)}}{\log m_\ell}\in[0,1].
\end{equation}
For an isolated template vertex the limit in \eqref{eq:template-log-grade} is zero.

The vector $\alpha=(\alpha_i)_{i\in I}$ satisfies the required edge inequalities on $T$. Indeed, for every $f\in E(T)$, the term of $Q_T(\mathbf n^{(\ell)})$ indexed by $f$ is at most $m_\ell$. Taking logarithms, dividing by $\log m_\ell$, and using \eqref{eq:template-log-grade}, we obtain
\[
\alpha(f)=\sum_{i\in f}\alpha_i=\lim_{\ell\to\infty}\frac{\log\prod_{i\in f}n_i^{(\ell)}}{\log m_\ell}\le1.
\]
Thus $\alpha$ satisfies every edge inequality required of an exponent vector.

Let $Z=\{i\in I:\alpha_i=0\}$. For each $\ell$, define coefficients by $b_i^{(\ell)}=n_i^{(\ell)}$ for $i\in Z$ and $s_i^{(\ell)}=n_i^{(\ell)}/m_\ell^{\alpha_i}$ for $i\notin Z$, and let $D_\ell=(\alpha,b^{(\ell)},s^{(\ell)})$. Every $b_i^{(\ell)}$ is a positive integer and every $s_i^{(\ell)}$ is a positive real number, as required for a triple on $T$.

The edge constraint on $\mathbf n^{(\ell)}$ makes each $D_\ell$ feasible. To verify this directly, let $w_i^{(\ell)}=b_i^{(\ell)}$ on $Z$ and $w_i^{(\ell)}=s_i^{(\ell)}$ outside $Z$. If $f$ is critical for $\alpha$, then $\sum_{i\in f}\alpha_i=1$, and hence
\[
\prod_{i\in f}w_i^{(\ell)}=\frac{\prod_{i\in f}n_i^{(\ell)}}{m_\ell}.
\]
Summing over the critical edges, $B_T(D_\ell)\le Q_T(\mathbf n^{(\ell)})/m_\ell\le1$.

For a homomorphism $\theta\in\operatorname{Hom}(H,T)$, write $g(\theta)=\sum_i k_i(\theta)\alpha_i$. By Lemma~\ref{lem:state-grade}, $g(\theta)\le\rho$. If $g(\theta)=\rho$, then no normalization by a power of $m_\ell$ is needed for the coordinates with $\alpha_i=0$. Since $(n)_k\le n^k$, we have
\[
\frac{\prod_{i\in I}(n_i^{(\ell)})_{k_i(\theta)}}{m_\ell^\rho}\le\left(\prod_{i\in Z}(b_i^{(\ell)})_{k_i(\theta)}\right)\left(\prod_{i\notin Z}(s_i^{(\ell)})^{k_i(\theta)}\right).
\]
The expression on the right is precisely the summand of $\Phi_{H,T}(D_\ell)$ indexed by $\theta$. Consequently, the sum of the normalized contributions of all homomorphisms with $g(\theta)=\rho$ is at most $\Phi_{H,T}(D_\ell)\le c_H(T)$.

It remains to bound the homomorphisms with $g(\theta)<\rho$. If at least one such homomorphism exists, the following minimum is positive because $\operatorname{Hom}(H,T)$ is finite
\[
\delta=\min\{\rho-g(\theta):\theta\in\operatorname{Hom}(H,T)\text{ and }g(\theta)<\rho\}>0.
\]
Set $a_i^{(\ell)}=\log n_i^{(\ell)}/\log m_\ell$. Since $\sum_i k_i(\theta)=h$ for every homomorphism, coordinatewise convergence permits us to choose $\ell_0$ such that $\max_i|a_i^{(\ell)}-\alpha_i|\le\delta/(2h)$ whenever $\ell\ge\ell_0$. Every nonleading homomorphism and every such $\ell$ then satisfy $\sum_i k_i(\theta)a_i^{(\ell)}\le g(\theta)+\sum_i k_i(\theta)|a_i^{(\ell)}-\alpha_i|\le g(\theta)+\delta/2\le\rho-\delta/2$. Applying the falling-factorial bound, we obtain
\[
0\le\frac{\prod_{i\in I}(n_i^{(\ell)})_{k_i(\theta)}}{m_\ell^\rho}\le m_\ell^{-\delta/2}.
\]
Since the number of homomorphisms is fixed, their total normalized contribution tends to zero. If every homomorphism satisfies $g(\theta)=\rho$, the nonleading sum is empty and the same conclusion is immediate.

By Lemma~\ref{lem:exact-template-count}, $P_{H,T}(\mathbf n^{(\ell)})$ is a sum over homomorphisms, and the preceding estimates bound each summand. Thus, $\limsup_{\ell\to\infty} A_{H,T}(m_\ell)/m_\ell^\rho\le c_H(T)$. Since $(m_\ell)$ was chosen to realize the limsup of the full normalized sequence, the desired upper bound follows.
\end{proof}

For each $L$, there are only finitely many nonisomorphic templates with at most $L$ vertices. Proposition~\ref{prop:fixed-template-limit} can therefore be applied with one common threshold to every such template; this is the precise uniformity needed after compression.

\begin{corollary}\label{cor:bounded-template-uniformity}
For every integer $L\ge0$ and every $\zeta>0$, there is an integer $m_0=m_0(H,r,L,\zeta)$ such that the following holds for every $m\ge m_0$. If $G$ is a complete blow-up of a template with at most $L$ vertices and $e(G)\le m$, then
\[
M_r(G,H)\le\bigl(\Lambda_r(H)+\zeta\bigr)m^\rho.
\]
\end{corollary}

\begin{proof}
Delete every zero-sized class from a blow-up representation of $G$. The remaining representation is a complete blow-up $T(\mathbf n)$ of an induced subtemplate with at most $L$ vertices, with every vertex class nonempty. Choose one representative of each isomorphism class of simple $r$-graphs on at most $L$ vertices; this is a finite collection. By Proposition~\ref{prop:fixed-template-limit}, we may choose a convergence threshold for each representative. Let $m_0$ be the maximum of these finitely many thresholds and $\binom Lr$, so that $m\ge m_0$ also implies $m\ge e(T)$ for every template under consideration. It follows that $A_{H,T}(m)\le(c_H(T)+\zeta)m^\rho$ simultaneously for all of them. The asserted estimate now follows from Lemma~\ref{lem:exact-template-count}, the definition of $A_{H,T}(m)$, and the inequality $c_H(T)\le\Lambda_r(H)$.
\end{proof}

To apply the corollary, for every fixed $0<\varepsilon<1$ we need to replace each host $G$ by a complete blow-up of a template of bounded order with at most $e(G)$ edges and at least $(1-\varepsilon)M_r(G,H)$ embeddings of $H$. We obtain such a replacement by first separating the vertices of each retained copy by random colors, as in the color-coding method of Alon, Yuster and Zwick~\cite{AlonYusterZwick1995}, and then compressing one color class at a time.

\subsection{Globally Rainbow Compression}

Let $J$ be an $r$-graph with a vertex coloring $\chi:V(J)\to[K]$ such that the vertices of every edge have distinct colors. An embedding $\phi\in\Emb_r(H,J)$ is \emph{globally rainbow} if $\chi\circ\phi$ is injective on $V(H)$. Write $R_\chi(J,H)$ for the number of globally rainbow embeddings. This edgewise condition does not guarantee that an embedding is globally rainbow, since two vertices of $H$ not contained in a common edge may receive the same color.

For a vertex $x$ in a color class $C$, its \emph{link} is the family of $(r-1)$-sets $A\subseteq V(J)\setminus C$ for which $A\cup\{x\}\in E(J)$. Replacing $x$ by $q$ clones means introducing a set of $q$ new vertices of the same color and assigning this same link to each of them. Since every edge meets $C$ at most once, this prescription determines all edges involving the new vertices and creates a finite simple $r$-graph.

We next quantify the loss of globally rainbow embeddings when compressing a single color class. The key is to compare the degree of each vertex with the number of globally rainbow embeddings containing it.

\begin{lemma}\label{lem:one-color-compression}
Let $0<\eta<1$, and let $C$ be a color class of $J$. There is a replacement of $C$ by at most $\lfloor\eta^{-1}\rfloor+1$ nonempty clone classes such that the resulting colored $r$-graph $J'$ satisfies $e(J')\le e(J)$ and $R_\chi(J',H)\ge(1-\eta)R_\chi(J,H)$.
\end{lemma}

\begin{proof}
For $x\in C$, let $d_x$ be the degree of $x$ in $J$, and let $a_x$ be the number of globally rainbow embeddings whose image contains $x$. Such an embedding contains at most one vertex of $C$. If $R^\circ$ denotes the number of globally rainbow embeddings avoiding $C$, then $R_\chi(J,H)=R^\circ+\sum_{x\in C}a_x$. Similarly, every edge meets $C$ at most once, so $e(J)=e(J-C)+\sum_{x\in C}d_x$. If $a_x>0$, choose an embedding counted by $a_x$ and a vertex of $H$ mapped to $x$. That vertex belongs to an edge of $H$, and the image of this edge is incident with $x$; hence $d_x>0$.

More generally, for each $x\in C$, choose an integer $q_x\ge0$ and replace $x$ by $q_x$ clones, deleting $x$ when $q_x=0$. Since every globally rainbow embedding uses at most one vertex of each color, the natural projection onto $J$ maps every globally rainbow embedding in the new graph to a globally rainbow embedding in $J$. Conversely, a globally rainbow embedding in $J$ using $x\in C$ gives rise to $q_x$ globally rainbow embeddings in the new graph, one for each clone of $x$. Thus, for each $x\in C$, there is a bijection between globally rainbow embeddings in the new graph using a clone of $x$ and pairs consisting of a globally rainbow embedding in $J$ using $x$ and a choice of one of its $q_x$ clones. The number of globally rainbow embeddings in the new graph is therefore $R^\circ+\sum_xq_xa_x$, while its edge count is $e(J-C)+\sum_xq_xd_x$.

Let $A=\sum_{x\in C}a_x$. If $A=0$, delete the whole class $C$. This leaves all $R^\circ=R_\chi(J,H)$ globally rainbow embeddings and cannot increase the edge count. Assume henceforth that $A>0$, and retain one clone of each vertex in
\[
X=\{x\in C:a_x\ge\eta A\}.
\]
Since $|X|\eta A\le A$, we have $|X|\le\lfloor\eta^{-1}\rfloor$.

Let $A_*=\sum_{x\in C\setminus X}a_x$ and $D_*=\sum_{x\in C\setminus X}d_x$. If $A_*=0$, the retained vertices indexed by $X$ preserve every globally rainbow embedding that uses $C$, and the lemma follows. Suppose that $A_*>0$. Then $D_*>0$. Vertices with $d_x=0$ also have $a_x=0$, and therefore $A_*=\sum_{x\in C\setminus X,\ d_x>0}d_x(a_x/d_x)$, whereas $D_*=\sum_{x\in C\setminus X,\ d_x>0}d_x$. The largest ratio $a_x/d_x$ among these vertices is at least their $d_x$-weighted average $A_*/D_*$. Choose $y\in C\setminus X$ with $d_y>0$ and $a_y/d_y\ge A_*/D_*$, and add $q=\lfloor D_*/d_y\rfloor$ clones of $y$. Since $D_*\ge d_y$, this integer is positive. Its contribution to the rainbow count satisfies
\[
qa_y\ge\left(\frac{D_*}{d_y}-1\right)a_y\ge A_*-a_y>A_*-\eta A.
\]
The last inequality uses $y\notin X$. On the other hand, $qd_y\le D_*$. Keeping one clone of every vertex in $X$ and $q$ clones of $y$ therefore does not increase the number of edges incident with $C$, and it retains more than $A-\eta A$ of the rainbow embeddings using $C$. Embeddings avoiding $C$ are unchanged, so the resulting count is at least $R^\circ+(1-\eta)A\ge(1-\eta)R_\chi(J,H)$. The number of nonempty clone classes is at most $|X|+1\le\lfloor\eta^{-1}\rfloor+1$.
\end{proof}

Applying the preceding lemma successively requires one further verification. Clone classes created in an earlier color must remain classes of vertices with identical incidence after later colors are compressed. This follows because a later clone receives the entire link of its representative, including all choices of vertices from the earlier clone classes.

\begin{proposition}\label{prop:bounded-template-replacement}
For every $\varepsilon>0$, there is an integer $L=L(H,r,\varepsilon)$ such that every finite simple $r$-graph $G$ has a complete blow-up $G'$ of a template with at most $L$ vertices for which $e(G')\le e(G)$ and $M_r(G',H)\ge(1-\varepsilon)M_r(G,H)$.
\end{proposition}

\begin{proof}
It is enough to consider $0<\varepsilon<1$, because the conclusion for such a value also implies the weaker conclusion for every larger value. Choose an integer $K\ge h$ such that $\binom h2/K\le\varepsilon/2$, and color every vertex of $G$ independently and uniformly with a color in $[K]$. For a fixed embedding $\phi$, the probability that its $h$ image vertices receive distinct colors is
$
p_K=\frac{(K)_h}{K^h}.
$
By the union bound,
$
1-p_K\le\frac{\binom h2}{K}\le\frac{\varepsilon}{2}.
$


Let $R$ be the number of globally rainbow embeddings. By linearity of expectation, $\mathbb E R=p_KM_r(G,H)$, so some coloring has at least $p_KM_r(G,H)$ such embeddings. Fix such a coloring and obtain $J_0$ by deleting every edge containing two vertices of the same color. Since every globally rainbow embedding survives, we have $e(J_0)\le e(G)$ and $R_\chi(J_0,H)\ge p_KM_r(G,H)$, while every edge of $J_0$ has vertices of distinct colors.

Set $\eta=\varepsilon/(2K)$, and apply Lemma~\ref{lem:one-color-compression} successively to the $K$ color classes. Cloning vertices within one color preserves the condition that every edge has vertices of distinct colors. Let $J_j$ be the graph after the first $j$ classes have been compressed. By Lemma~\ref{lem:one-color-compression}, we have $R_\chi(J_j,H)\ge(1-\eta)R_\chi(J_{j-1},H)$ for $1\le j\le K$. Iterating, we obtain
\[
R_\chi(J_K,H)\ge(1-\eta)^KR_\chi(J_0,H)\ge(1-\eta)^Kp_KM_r(G,H).
\]
Also, we have $(1-\eta)^K\ge1-K\eta=1-\varepsilon/2$ and $p_K\ge1-\varepsilon/2$. Hence $R_\chi(J_K,H)\ge(1-\varepsilon/2)^2M_r(G,H)\ge(1-\varepsilon)M_r(G,H)$.

We next prove that $J_K$ is a complete blow-up. 
We prove by induction that, after the first $j$ colors have been compressed, all vertices within each resulting clone class have the same link in $J_j$. This is clear for $j=1$. Suppose it holds in $J_j$, and consider the compression of color $j+1$. All clones of a given representative have the same link, so the new classes satisfy the required property. For an earlier class, edges avoiding color $j+1$ are unchanged. For edges containing a new clone, we may replace it by its representative in $J_j$, apply the induction hypothesis to the earlier class, and then replace the representative by the same new clone. Thus the property is preserved at every step, and all clone classes in $J_K$ have the same link within each class.

Let $T$ be the template whose vertices are the nonempty clone classes, with $\{C_1,\ldots,C_r\}\in E(T)$ if and only if there exist vertices $x_i\in C_i$ such that $\{x_1,\ldots,x_r\}\in E(J_K)$. Since every edge of $J_K$ has vertices of distinct colors, it meets each clone class in at most one vertex. By the preceding induction, if one choice of vertices from $C_1,\ldots,C_r$ forms an edge, then every such choice does. Consequently, $J_K$ is exactly the complete blow-up of $T$, with vertex classes given by the clone classes.

Lemma~\ref{lem:one-color-compression} creates at most $\lfloor\eta^{-1}\rfloor+1$ classes in each color. Therefore $|V(T)|\le L$, where
$
L=K\left(\left\lfloor\frac{2K}{\varepsilon}\right\rfloor+1\right).
$
Finally, $M_r(J_K,H)\ge R_\chi(J_K,H)$ because globally rainbow embeddings are among all embeddings. By Lemma~\ref{lem:one-color-compression} and the construction of $J_0$, we have $e(J_K)\le e(J_0)\le e(G)$. Taking $G'=J_K$ completes the proof.
\end{proof}

Proposition~\ref{prop:bounded-template-replacement} is uniform in the number of host edges. Once $H$, $r$ and $\varepsilon$ are fixed, the order of the template is bounded by the displayed integer $L$ for every host. Corollary~\ref{cor:bounded-template-uniformity} can consequently be applied to all the resulting blow-ups with a single threshold in $m$.

\subsection{Completion of the Proof}

For each feasible triple, we can complete the hosts in the realization construction to exactly $m$ edges. By Proposition~\ref{prop:variational-bounds}, the lower limit of the normalized extremal count is at least $\Lambda_r(H)$. The reverse inequality with the same constant follows from bounded-template replacement and the fixed-template limit.

\begin{proof}[Proof of Theorem~\ref{thm:hypergraph-main}]
By Proposition~\ref{prop:variational-bounds}, we have $0<\Lambda_r(H)\le(r!)^\rho$ and
\[
\liminf_{m\to\infty}\frac{M_r(m,H)}{m^\rho}\ge\Lambda_r(H).
\]
It remains to prove the matching upper bound.

Fix $0<\varepsilon<1$, and let $L=L(H,r,\varepsilon)$ be supplied by Proposition~\ref{prop:bounded-template-replacement}. For an arbitrary $m$-edge $r$-graph $G$, we can therefore choose a complete blow-up $G'$ of a template with at most $L$ vertices such that
$
(1-\varepsilon)M_r(G,H)\le M_r(G',H)
$
and
$\qquad e(G')\le m.
$
Fix also $\zeta>0$. By Corollary~\ref{cor:bounded-template-uniformity}, there exists a threshold $m_0=m_0(H,r,L,\zeta)$, independent of $G$, such that $M_r(G',H)\le(\Lambda_r(H)+\zeta)m^\rho$ whenever $m\ge m_0$. Combining the two inequalities and maximizing over all $m$-edge hosts, we obtain
$
(1-\varepsilon)\frac{M_r(m,H)}{m^\rho}\le\Lambda_r(H)+\zeta
$
for every integer $m\ge m_0$.
Letting $m\to\infty$ and then $\zeta\to0$, we obtain $\limsup_{m\to\infty}M_r(m,H)/m^\rho\le\Lambda_r(H)/(1-\varepsilon)$. Finally, letting $\varepsilon\to0$ gives
\[
\limsup_{m\to\infty}\frac{M_r(m,H)}{m^\rho}\le\Lambda_r(H).
\]
The upper and lower limits are therefore equal, and the convergence holds along the full sequence of positive integers.

Every unlabelled copy of $H$ has exactly $|\Aut(H)|$ labelled embeddings. Thus
\[
\lim_{m\to\infty}\frac{N_r(m,H)}{m^\rho}=\frac{\Lambda_r(H)}{|\Aut(H)|}=\lambda_r(H).
\]
Proposition~\ref{prop:bounded-template-replacement} is exactly the final assertion in the theorem, so all parts of Theorem~\ref{thm:hypergraph-main} follow.
\end{proof}
\section{Proofs of Corollary~\ref{cor:alon-first} and Theorem~\ref{thm:alon-second}}\label{sec:graph-case}

\subsection{Fractional Independence and Finite Cores}\label{sec:graph-definitions}

For a graph $H$ without isolated vertices, the dual program for $\rho^*(H)$ is the fractional independent-set program from Section~\ref{sec:introduction}, so $\alpha^*(H)=\rho^*(H)$. The following identity relates this optimum to the neighborhood deficiency $D(H)$. For a graph $J$ and a vertex $v\in V(J)$, write $d_J(v)=|N_J(v)|$.

\begin{proposition}\label{prop:graph-exponent}
We have $\alpha^*(H)=(|V(H)|+D(H))/2=\gamma(H)$.
\end{proposition}

For nonempty $H$, Alon~\cite[Theorem~5]{Alon1981} proved that $N(m,H)=\Theta_H(m^{\gamma(H)})$, while Friedgut and Kahn~\cite[main theorem]{FriedgutKahn1998} proved that $M_2(m,H)=\Theta_H(m^{\rho^*(H)})$. Since $M_2(m,H)=|\Aut(H)|N(m,H)$, the exponents agree. By linear-programming duality, we then have $\gamma(H)=\rho^*(H)=\alpha^*(H)$. For the empty graph, all three quantities are zero.

\begin{lemma}[Nemhauser and Trotter~{\cite{NemhauserTrotter1974,NemhauserTrotter1975}}]\label{thm:nt-half-integrality}
Let $J$ be a finite simple graph. Every extreme point of the polytope
\[
  \{x\in[0,1]^{V(J)}:x_u+x_v\le1\text{ for every }uv\in E(J)\}
\]
has all its coordinates in $\{0,\tfrac12,1\}$.
\end{lemma}

Since $H$ has no isolated vertices, every fractional independent set of $H$ belongs to $[0,1]^{V(H)}$. The linear objective $\sum_{v\in V(H)}x_v$ attains its maximum at an extreme point of the feasible polytope. Thus, by Lemma~\ref{thm:nt-half-integrality}, we may choose an optimal fractional independent set with coordinates in $\{0,\tfrac12,1\}$. These values correspond to the fixed-size classes, the complete graph of order $\sqrt m$ and the linear-size classes in the finite-core construction. Zeng~\cite{Zeng2026FiniteKernel} studies related constructions for graph counting with given numbers of vertices and edges.

For each admissible set $S$ and each embedding $\psi:H[S]\hookrightarrow F$, the corresponding term in $\Phi_H^{\mathrm{core}}(F,y)$ is
\[
  \left(\prod_{z\in Z_S}Y_{\Gamma_\psi(z)}\right)
  [2(1-s(y))]^{|V(B_S)|/2}.
\]
The map $\psi$ fixes the images of the vertices in $S$. For each $z\in Z_S$, the factor $Y_{\Gamma_\psi(z)}$ sums the weights $y_A$ over the nonempty sets $A\subseteq V(F)$ containing $\Gamma_\psi(z)$. The last factor accounts for the vertices of $B_S$ embedded in the clique. We next compare $|Z_S|+\alpha^*(B_S)$ with $\alpha^*(H)$.

Write $h=|V(H)|$ and $\alpha=\alpha^*(H)=\rho^*(H)$. For $S\subseteq V(H)$, we use the sets $R_S$ and $Z_S$, the graph $B_S$, the quantity $\beta(S)$, and the terms balanced and admissible as defined in Section~\ref{sec:introduction}. In particular, $B_S$ is empty or has no isolated vertices.

The quantity $\beta(S)$ is the exponent contributed by the vertices in $Z_S$ and the graph $B_S$. It can never exceed $\alpha^*(H)$, so only equality cases can contribute to the leading constant.

\begin{lemma}\label{lem:residual-grade}
For every $S\subseteq V(H)$, one has $\beta(S)\le\alpha$.
\end{lemma}

\begin{proof}
Choose an optimal fractional independent set $x$ on $B_S$, taking the empty vector if $B_S$ is empty. Since $B_S$ has no isolated vertices, every coordinate of $x$ is at most $1$. Extend $x$ to $V(H)$ by assigning $0$ on $S$ and $1$ on $Z_S$. The extended vector is feasible because vertices of $Z_S$ have no neighbors in $R_S$, edges within $B_S$ satisfy the constraints for $x$, and edges meeting $S$ have total weight at most $1$. Its objective value is $|Z_S|+\alpha^*(B_S)=\beta(S)$, so $\beta(S)\le\alpha$.
\end{proof}

For a leading homomorphism into a graph template with a given exponent vector, the inverse image $S$ of the vertices with exponent zero will be shown to be admissible in Proposition~\ref{prop:graph-variational-equivalence}.


For a fixed core, the supremum over its weights is attained.

\begin{lemma}\label{lem:fixed-core-compactness}
For every fixed finite graph $F$, the feasible set of vectors $y$ is compact, the function $\Phi_H^{\mathrm{core}}(F,y)$ is continuous on that set, and its supremum over $y$ is a finite maximum.
\end{lemma}

\begin{proof}
The feasible set is nonempty and compact, since it has finitely many coordinates and is defined by $y_A\ge0$ and $\sum_A|A|y_A\le1$. The objective $\Phi_H^{\mathrm{core}}(F,y)$ is continuous on this set, including at $s(y)=1$ under the convention for empty $B_S$. It therefore attains a finite maximum.
\end{proof}

If $V(H)=\varnothing$, there is exactly one map from $V(H)$ into every host, so $M_2(m,H)=N(m,H)=1$, while $\gamma(H)=\rho^*(H)=0$. In both variational formulas, the only term corresponds to the empty map and equals $1$ by the empty-product convention; hence $\Lambda_2(H)=\Lambda_{\mathrm{core}}(H)=1$, and Corollary~\ref{cor:alon-first} follows. For the remaining finite-core estimates, assume that $H$ is nonempty.

\subsection{Comparison of the Variational Constants}

To compare $\Lambda_2(H)$ with $\Lambda_{\mathrm{core}}(H)$, we first bound the total weighted contribution of homomorphisms from a balanced graph. Only critical edges of the template are needed for this bound, since every leading homomorphism maps each edge assigned positive weight by an optimal fractional edge cover to a critical edge.

\begin{lemma}\label{lem:balanced-state-bound}
Let $B$ be a nonempty balanced graph without isolated vertices, let $T$ be a finite graph with vertex set $I$, and let $a=(a_i)_{i\in I}$ be an exponent vector on $T$. For positive real weights $(s_i)_{i\in I}$, let $q=\sum_{ij\in E(T),\,a_i+a_j=1}s_i s_j$. Then
\[
\sum_{\substack{\theta\in\operatorname{Hom}(B,T)\\\sum_{v\in V(B)}a_{\theta(v)}=|V(B)|/2}}\ \prod_{v\in V(B)}s_{\theta(v)}\le(2q)^{|V(B)|/2}.
\]
\end{lemma}

\begin{proof}
Write $d=|V(B)|$ and let $\lambda$ be an optimal fractional edge cover of $B$. Since $B$ is balanced, $\sum_e\lambda_e=\rho^*(B)=d/2$, so every homomorphism in the displayed sum is leading with respect to $a$. By Lemma~\ref{lem:state-grade} with exponent vector $a$, each such homomorphism maps every edge $e$ with $\lambda_e>0$ to a critical edge of $T$. If the sum is empty, the bound holds; otherwise $q>0$.

Suppose first that all $s_i$ are rational. Choose a positive integer $n$ such that every $ns_i$ is an integer, and take pairwise disjoint sets $L_i$ of $ns_i$ labels. Let $\mathcal F$ be the family of assignments obtained by choosing a homomorphism $\theta$ in the sum and then a label in $L_{\theta(v)}$ for each $v\in V(B)$. The labels determine $\theta$, so
\[
|\mathcal F|=n^d\sum_{\substack{\theta\in\operatorname{Hom}(B,T)\\\sum_v a_{\theta(v)}=d/2}}\ \prod_{v\in V(B)}s_{\theta(v)}.
\]
For every $e=uv$ with $\lambda_e>0$, the projection onto $u,v$ has at most $2n^2q$ values, since the two labels belong to the classes of an oriented critical edge. Applying Lemma~\ref{lem:weighted-projection} to these edges with weights $\lambda_e$, we obtain $|\mathcal F|\le(2n^2q)^{d/2}$. Division by $n^d$ proves the bound for rational weights. The set of homomorphisms in the sum depends only on $B$, $T$ and $a$, and not on the weights. Both sides are continuous in the weights, so rational approximation proves the bound for positive real weights.
\end{proof}

The vertices with exponent zero in a graph template define a finite core. Expanding the falling factorials groups the contributions of leading homomorphisms by embeddings into this core, and their sum is bounded by the core objective. Conversely, every core objective value is a limit of sums of contributions of leading homomorphisms for feasible triples.

\begin{proposition}\label{prop:graph-variational-equivalence}
For every finite simple graph $H$ without isolated vertices, one has $\Lambda_2(H)=\Lambda_{\mathrm{core}}(H)$.
\end{proposition}

\begin{proof}
The empty graph was treated above, so assume that $H$ is nonempty. Fix a finite graph template $T$ with vertex set $I$ and a feasible triple $D=(a,b,s)$, where $a=(a_i)_{i\in I}$ is its exponent vector. No homomorphism from $H$ to $T$ maps a vertex to an isolated vertex of $T$, and isolated vertices contribute no term to $B_T(D)$. Deleting them therefore does not change $\Phi_{H,T}(D)$ or $B_T(D)$. We may therefore assume that $T$ has no isolated vertices. If $T$ is empty, then $\Phi_{H,T}(D)=0\le\Lambda_{\mathrm{core}}(H)$. Otherwise every $a_i$ lies in $[0,1]$. Partition $I$ into $I_0=\{i:a_i=0\}$, $I_1=\{i:a_i=1\}$ and $I_*=\{i:0<a_i<1\}$.

To identify the terms of the core objective, fix a leading homomorphism $\theta\in\operatorname{Hom}(H,T)$, put $x_v=a_{\theta(v)}$ for $v\in V(H)$, and let $S=\theta^{-1}(I_0)$. The vector $x$ is an optimal fractional independent set of $H$, and its restriction to $B_S$ is feasible. By Lemma~\ref{lem:residual-grade}, we have
\[
\alpha=\sum_{z\in Z_S}x_z+\sum_{v\in V(B_S)}x_v\le |Z_S|+\alpha^*(B_S)=\beta(S)\le\alpha.
\]
Equality holds throughout. Consequently $x_z=1$ for every $z\in Z_S$, and the restriction of $x$ to $B_S$ is optimal and strictly positive. If $B_S$ is nonempty, let $\mu$ be an optimal fractional edge cover of $B_S$. The homomorphism $\theta|_{V(B_S)}$ satisfies $\sum_{v\in V(B_S)}a_{\theta(v)}=\alpha^*(B_S)=\rho^*(B_S)$. Since every $x_v$ is positive, it follows from Lemma~\ref{lem:state-grade} that $\sum_{e\ni v}\mu_e=1$ for every $v\in V(B_S)$. Summing over the vertices, we obtain $2\alpha^*(B_S)=2\sum_e\mu_e=|V(B_S)|$. Thus $S$ is admissible. Every vertex of $B_S$ has a neighbor in $B_S$ with positive $x$-value, so $0<x_v<1$ on $V(B_S)$. Therefore $\theta$ maps $Z_S$ into $I_1$ and $V(B_S)$ into $I_*$.

Let $F$ be the complete blow-up of $T[I_0]$ with pairwise disjoint vertex classes $C_i$ of sizes $b_i$ for $i\in I_0$. Every neighbor of a vertex $j\in I_1$ belongs to $I_0$. Define $A_j=\bigcup_{i\in N_T(j)}C_i$, which is nonempty because $T$ has no isolated vertices, and for every nonempty $A\subseteq V(F)$ set $y_A=\sum_{j\in I_1:\,A_j=A}s_j$. Empty sums are zero. Then
\[
s(y)=\sum_{\substack{ij\in E(T)\\i\in I_0,\ j\in I_1}}b_i s_j.
\]
Put $q=\sum_{ij\in E(T[I_*]),\,a_i+a_j=1}s_i s_j$. Every critical edge joins $I_0$ to $I_1$ or has both endpoints in $I_*$, so $s(y)+q=B_T(D)\le1$. In particular, $y$ is feasible for $F$.

For each leading homomorphism $\theta$, expand $\prod_{i\in I_0}(b_i)_{k_i(\theta)}$ as the number of choices of injections from $\theta^{-1}(i)$ into $C_i$. Their union is an embedding $\psi\in\Emb_2(H[S],F)$. Conversely, the classes containing the images of $\psi$ determine $\theta|_S$ and these injections uniquely. Thus the terms of $\Phi_{H,T}(D)$ can be grouped by pairs $(S,\psi)$ with $S$ admissible and $\psi\in\Emb_2(H[S],F)$. For a fixed pair, a vertex $z\in Z_S$ may be mapped to $j\in I_1$ precisely when $A_j\supseteq\Gamma_\psi(z)$, so the sum of its weights is $Y_{\Gamma_\psi(z)}$. These choices are independent, since there are no edges from $Z_S$ to $R_S$ and homomorphisms need not be globally injective. The restriction of $\theta$ to $B_S$ satisfies $\sum_{v\in V(B_S)}a_{\theta(v)}=|V(B_S)|/2$. Discarding its adjacency requirements to $S$ enlarges the set of allowed homomorphisms, and Lemma~\ref{lem:balanced-state-bound}, applied to $B_S$ and $T[I_*]$, bounds their total weight by $(2q)^{|V(B_S)|/2}$. When $B_S$ is empty this factor is $1$. Summing over the pairs and using $q\le1-s(y)$, we obtain
\[
\Phi_{H,T}(D)\le\sum_{\substack{S\subseteq V(H)\\S\text{ admissible}}}\ \sum_{\psi\in\Emb_2(H[S],F)}\left(\prod_{z\in Z_S}Y_{\Gamma_\psi(z)}\right)(2q)^{|V(B_S)|/2}\le\Phi_H^{\mathrm{core}}(F,y).
\]
Taking the supremum over $T$ and $D$ proves $\Lambda_2(H)\le\Lambda_{\mathrm{core}}(H)$.

For the reverse inequality, fix a finite graph $F$ and a feasible vector $y$, and put $c=1-s(y)$. Begin a template with $F$, assigning exponent $0$ and integer parameter $1$ to every vertex of $F$. For every nonempty $A\subseteq V(F)$ with $y_A>0$, add a vertex $p_A$ with exponent $1$ and coefficient $y_A$, adjacent to precisely the vertices of $A$. Suppose first that $c>0$. Choose an integer $t\ge\max\{2,h\}$ and add a clique $Q_t$ on $t$ vertices, each with exponent $1/2$ and coefficient $w_t=\sqrt{2c/(t(t-1))}$. Join every vertex of $Q_t$ to every vertex of $F$, and add no further edges. Denote the resulting template and triple by $T_t$ and $D_t$. The sum of exponents on every template edge is at most $1$, and $B_{T_t}(D_t)=s(y)+\binom{t}{2}w_t^2=1$.

Fix an admissible set $S$ and $\psi\in\Emb_2(H[S],F)$. Map $S$ according to $\psi$, map each $z\in Z_S$ to any $p_A$ with $A\supseteq\Gamma_\psi(z)$, and map $V(B_S)$ injectively into $Q_t$. Every edge of $H[S]$ is preserved by $\psi$, the required edges from $Z_S$ to $S$ are present by the choice of $p_A$, and all required edges within $B_S$ and between $B_S$ and $S$ are present by the construction of $Q_t$. There are no edges from $Z_S$ to $R_S$. Thus every map so obtained is a homomorphism. Its total exponent is $|Z_S|+|V(B_S)|/2=\alpha$, and all factors belonging to vertices with exponent zero are $1$. The families of homomorphisms arising from distinct pairs $(S,\psi)$ are disjoint because the inverse image of $V(F)$ recovers $S$ and the restriction to $S$ recovers $\psi$. Consequently,
\[
\Phi_{H,T_t}(D_t)\ge\sum_{\substack{S\subseteq V(H)\\S\text{ admissible}}}\ \sum_{\psi\in\Emb_2(H[S],F)}\left(\prod_{z\in Z_S}Y_{\Gamma_\psi(z)}\right)(t)_{|V(B_S)|}w_t^{|V(B_S)|}.
\]
For every fixed nonnegative integer $d$, one has $(t)_d w_t^d\to(2c)^{d/2}$ as $t\to\infty$, with both sides equal to $1$ when $d=0$. Since the sums are finite for fixed $F$ and $H$, we may pass to the limit in $\Lambda_2(H)\ge\Phi_{H,T_t}(D_t)$ to obtain $\Lambda_2(H)\ge\Phi_H^{\mathrm{core}}(F,y)$.

If $c=0$, omit the clique and retain only $F$ and the vertices $p_A$. For the resulting template, the sum of weight products over the critical edges is $s(y)=1$. Every term of $\Phi_H^{\mathrm{core}}(F,y)$ with nonempty $B_S$ is zero, while the same choices for $S$ and $Z_S$ realize all terms with empty $B_S$ as homomorphisms into this template that are leading with respect to its exponent vector. Thus $\Lambda_2(H)\ge\Phi_H^{\mathrm{core}}(F,y)$ also when $c=0$. Taking the supremum over $F$ and $y$ proves the reverse inequality.
\end{proof}

By Propositions~\ref{prop:graph-variational-equivalence} and~\ref{prop:variational-bounds}, we have $0<\Lambda_{\mathrm{core}}(H)\le2^\alpha$. 
The graph carried by the bounded core affects the objective only through the set of available core embeddings. Adding core edges can enlarge that set and leaves all weights $y_A$ unchanged.

\begin{proposition}\label{prop:clique-core-reduction}
For every finite graph $H$ without isolated vertices,
\[
\Lambda_{\mathrm{core}}(H)=\sup_{t\ge0}\ \sup_y\Phi_H^{\mathrm{core}}(K_t,y),
\]
where the inner supremum is over the feasible vectors indexed by the nonempty subsets of $V(K_t)$.
\end{proposition}

\begin{proof}
Complete a finite core $F$ of order $t$ to $K_t$ on the same vertex set. The feasible vectors $y$ and the contribution of each embedding into $F$ remain unchanged, while $\Emb_2(H[S],F)\subseteq\Emb_2(H[S],K_t)$ for every admissible $S$. Thus $\Phi_H^{\mathrm{core}}(F,y)\le\Phi_H^{\mathrm{core}}(K_t,y)$. Taking suprema proves the assertion, since every $K_t$, including $K_0$, is an allowed core.
\end{proof}

The equality of the two variational constants identifies the graph limit in Theorem~\ref{thm:hypergraph-main}. By completing the core and maximizing the weights for each fixed core, we obtain the final form of Corollary~\ref{cor:alon-first}.

\begin{proof}[Proof of Corollary~\ref{cor:alon-first}]
By Proposition~\ref{prop:graph-exponent} and linear-programming duality, $\gamma(H)=\alpha^*(H)=\rho^*(H)$. Applying Theorem~\ref{thm:hypergraph-main} with $r=2$ and then Proposition~\ref{prop:graph-variational-equivalence}, we obtain
\[
\lim_{m\to\infty}\frac{N(m,H)}{m^{\gamma(H)}}=\frac{\Lambda_2(H)}{|\Aut(H)|}=\frac{\Lambda_{\mathrm{core}}(H)}{|\Aut(H)|}.
\]
Proposition~\ref{prop:clique-core-reduction} restricts the supremum to complete cores $K_t$. By Lemma~\ref{lem:fixed-core-compactness}, the supremum over feasible weights on each $K_t$ is a maximum. These statements prove the corollary, including the empty graph as treated above.
\end{proof}

\subsection{Star Forests and Theorem~\ref{thm:alon-second}}

For a star forest $H$, an extremal host can be chosen to be a star forest for every given number of edges. For a graph $G$, the vertices of its line graph are the edges of $G$, with two vertices adjacent if and only if the corresponding edges share an endpoint; see Diestel~\cite{Diestel2017}. Passing to the line graph converts copies of $H$ into induced copies of a disjoint union of cliques on a fixed number of vertices. Taking complements then produces a complete multipartite graph, to which the finite induced-copy theorem of Brown and Sidorenko applies.

For graphs $X$ and $T$, let $I(X,T)$ denote the number of sets $A\subseteq V(X)$ for which the induced graph $X[A]$ is isomorphic to $T$. The following finite extremal result is due to Brown and Sidorenko~\cite[Proposition~1]{BrownSidorenko1994}.

\begin{lemma}[Brown and Sidorenko]\label{lem:brown-sidorenko}
Let $T$ be a finite complete multipartite graph, and let $n\ge1$. There is an $n$-vertex complete multipartite graph $P$ such that
\[
I(P,T)=\max\{I(X,T):|V(X)|=n\}.
\]
\end{lemma}

For $|V(T)|\ge3$ and every $n\ge|V(T)|$, see the proof by Liu, Pikhurko, Sharifzadeh and Staden~\cite[proof of Lemma~1.2]{LiuPikhurkoSharifzadehStaden2023}. Their count is the number of vertex subsets inducing $T$, as in the definition of $I(X,T)$. If $n<|V(T)|$, all counts are zero. If $|V(T)|\le1$, the count is independent of $X$; for $|V(T)|=2$, a complete graph maximizes the count when $T=K_2$, and an empty graph does so when $T$ has no edges.

\begin{proof}[Proof of Theorem~\ref{thm:alon-second}]
If $V(H)=\varnothing$, take $S_0$ to be the empty graph and take $S_m=K_{1,m}$ when $m\ge1$. Every graph contains exactly one copy of the empty graph, so the assertion follows in this case. We may therefore assume that $H$ is nonempty, and write
$
H=\bigsqcup_{i=1}^q K_{1,a_i}
$
and
$
J=\bigsqcup_{i=1}^q K_{a_i},
$
where $q\ge1$ and $a_i\ge1$ for every $i$. In particular, $e(H)=|V(J)|=\sum_{i=1}^q a_i$.
If $m<e(H)$, then $N(m,H)=0$. Taking $S_0$ to be the empty graph and $S_m=K_{1,m}$ for $1\le m<e(H)$ gives the required conclusion. Hence we may assume $m\ge e(H)$.


Let $G$ be a graph with $e(G)=m$. Its line graph $L(G)$ has vertex set $E(G)$, and two vertices of $L(G)$ are adjacent precisely when the corresponding edges of $G$ have a common endpoint. In particular, $|V(L(G))|=m$.

For each copy $Q$ of $H$ in $G$, the edges of its star components correspond to pairwise nonadjacent cliques of orders $a_1,\ldots,a_q$ in $L(G)$. Hence $L(G)[E(Q)]\cong J$. Since $H$ has no isolated vertices, each copy $Q$ is uniquely determined by its edge set $E(Q)$. Thus distinct copies of $H$ give distinct induced copies of $J$ in $L(G)$, and
$
N(G,H)\le I(L(G),J).
$


The complement of $J$ is the complete multipartite graph $\overline J=K_{a_1,\ldots,a_q}$. By Lemma~\ref{lem:brown-sidorenko} with $n=m$, we may choose an $m$-vertex complete multipartite graph $P$ maximizing $I(P,\overline J)$. Let $B_1,\ldots,B_s$ be its nonempty parts, let $b_j=|B_j|$, and set $C=\overline P$. Then $\sum_{j=1}^s b_j=m$ and $C=\bigsqcup_{j=1}^sK_{b_j}$. For every graph $X$ on $m$ vertices and every $A\subseteq V(X)$, the identity $\overline X[A]=\overline{X[A]}$ implies that $X[A]\cong J$ if and only if $\overline X[A]\cong\overline J$. Therefore
\begin{equation}\label{1.4e1}
I(X,J)=I(\overline X,\overline J)\le I(P,\overline J)=I(C,J).
\end{equation}

For each $j\in\{1,\ldots,s\}$, let $R_j$ be a copy of $K_{1,b_j}$, with all the graphs $R_j$ pairwise vertex-disjoint, and define $S_m=\bigsqcup_{j=1}^sR_j$. Identify the $b_j$ vertices of the clique $C[B_j]$ with the $b_j$ edges of $R_j$. Two edges of $S_m$ have a common endpoint exactly when they belong to the same $R_j$, so this identification is an isomorphism from $C$ to $L(S_m)$. In particular,
$
e(S_m)=\sum_{j=1}^s b_j=m$
and
$
L(S_m)\cong C.
$

It remains to compare the two counts for $C$ and $S_m$. Under the identification $V(C)=E(S_m)$, fix $A\subseteq V(C)$ and write $t_j=|A\cap B_j|$. Let $Q_A$ be the subgraph of $S_m$ with edge set $A$ and vertex set consisting of its edge endpoints. Then
$
C[A]\cong\bigsqcup_{j:t_j>0}K_{t_j}$
and
$
Q_A\cong\bigsqcup_{j:t_j>0}K_{1,t_j}.
$
Consequently, $C[A]\cong J$ if and only if $Q_A\cong H$, since both conditions are equivalent to
$
\{\!\{t_j:t_j>0\}\!\}=\{\!\{a_i:1\le i\le q\}\!\}.
$

Since $Q_A$ is uniquely determined by $A$, and every copy of $H$ in $S_m$ has no isolated vertices, the map $A\mapsto Q_A$ gives a bijection between the vertex sets inducing $J$ in $C$ and the copies of $H$ in $S_m$. Hence
\begin{equation}\label{1.4e2}
    I(C,J)=N(S_m,H).
\end{equation}

Applying the extremal inequality for $C$ to $L(G)$ and \eqref{1.4e1}-\eqref{1.4e2}, we obtain
\[
N(G,H)\le I(L(G),J)\le I(C,J)=N(S_m,H)
\]
for every graph $G$ with $e(G)=m$.
Since $e(S_m)=m$, it follows that $N(m,H)=N(S_m,H)$, completing the proof.
\end{proof}
\section{Concluding Remarks}\label{sec:concluding-remarks}

The finite-core formula in Corollary~\ref{cor:alon-first} can be used to determine the leading constants for specific graphs. We illustrate this for paths. Fixed-edge path-counting problems were studied by Bollob\'as and Sarkar~\cite{BollobasSarkar2001,BollobasSarkar2003}. For an integer $\ell\ge2$, let $P_\ell$ denote the path on $\ell$ vertices. Paths of even order attain the projection bound asymptotically, while paths of odd order require an optimization over the core weights.

\begin{proposition}\label{prop:odd-path-constant}
For each positive integer $k$, define $c_k=\max_{a\in\mathbb Z,\,a\ge k}\frac{(a)_k}{a^{k+1}}$. Then $\rho^*(P_{2k})=k$ and $\rho^*(P_{2k+1})=k+1$. For fixed $k$, as $m\to\infty$,
\[
N(m,P_{2k})=2^{k-1}m^k+O_k(m^{k-1/2}),
\qquad
N(m,P_{2k+1})=\left(\frac{c_k}{2}+o(1)\right)m^{k+1}.
\]
Moreover, $\Lambda_{\mathrm{core}}(P_{2k+1})=c_k$. If $a$ attains the maximum defining $c_k$, this value is attained by the core $F=K_a$ with $y_{V(F)}=1/a$ and all other weights zero.
\end{proposition}

\begin{proof}
The maximum defining $c_k$ exists because its value at $a=k$ is positive and $(a)_k/a^{k+1}\le1/a\to0$. For $P_{2k}$, we sum the fractional independent-set constraints over a perfect matching to obtain $\alpha^*(P_{2k})\le k$, and the constant vector $1/2$ attains equality. Thus $\rho^*(P_{2k})=k$. Let $q$ be the largest integer with $\binom q2\le m$. Consider the disjoint union of $K_q$ and a matching of $m-\binom q2$ edges. Counting embeddings in the clique and applying Proposition~\ref{prop:global-projection-bound} for the upper bound, we obtain
\[
(q)_{2k}\le M_2(m,P_{2k})\le(2m)^k.
\]
Since $q=\sqrt{2m}+O(1)$, we have $(q)_{2k}=(2m)^k+O_k(m^{k-1/2})$. Dividing by the two automorphisms of the path proves the even-order estimate.

Write $P=P_{2k+1}=v_1\cdots v_{2k+1}$, and put $O=\{v_1,v_3,\ldots,v_{2k+1}\}$ and $E_0=\{v_2,v_4,\ldots,v_{2k}\}$. Every feasible fractional independent set $x$ satisfies $\sum_vx_v=\sum_{i=1}^k(x_{v_{2i-1}}+x_{v_{2i}})+x_{v_{2k+1}}\le k+1$, with equality attained by the indicator of $O$. Equality forces $x_{v_{2k+1}}=1$ and $x_{v_{2i-1}}+x_{v_{2i}}=1$ for every $i$; applying the intervening edge constraints backwards from $v_{2k+1}$ then forces $x$ to equal the indicator of $O$. Hence $\alpha^*(P)=\rho^*(P)=k+1$, with a unique optimal vector. For an admissible set $S$, we can use the construction in Lemma~\ref{lem:residual-grade} to obtain an optimal vector equal to $0$ on $S$, $1$ on $Z_S$ and $1/2$ on $V(B_S)$. Uniqueness therefore forces $B_S=\varnothing$, $Z_S=O$ and $S=E_0$. Conversely, deleting $E_0$ leaves the $k+1$ isolated vertices of $O$, so $E_0$ is admissible.

For a finite core $F$ and feasible weights $y$, put $p_u=Y_{\{u\}}$, so that $\sum_up_u=s(y)\le1$. Since $E_0$ is the unique admissible set and $P[E_0]$ has no edges, the objective is
\[
\Phi_P^{\mathrm{core}}(F,y)
=\sum_{\substack{u_1,\ldots,u_k\in V(F)\\
\text{\rm pairwise distinct}}}
p_{u_1}p_{u_k}\prod_{i=1}^{k-1}Y_{\{u_i,u_{i+1}\}}.
\]
For each tuple, choose a position $h$ where $p_{u_h}$ is minimal. Bound the factor $Y_{\{u_i,u_{i+1}\}}$ by $p_{u_{i+1}}$ when $i<h$ and by $p_{u_i}$ when $i\ge h$. The summand corresponding to this tuple is therefore at most $(\min_i p_{u_i})\prod_{i=1}^kp_{u_i}$, including when $k=1$.

The empty core has objective zero. Otherwise, label $V(F)=\{1,\ldots,n\}$ so that $p_1\ge\cdots\ge p_n$, and put $p_{n+1}=0$. Let $e_k$ denote the elementary symmetric polynomial of degree $k$, interpreted as zero when fewer than $k$ variables are present. For $j\ge k$, we apply Maclaurin's inequality (see Hardy, Littlewood and P\'olya~\cite{HardyLittlewoodPolya1952}). By continuity, it holds for nonnegative variables as well. Thus
\[
e_k(p_1,\ldots,p_j)
\le\binom jk\left(\frac{\sum_{i=1}^jp_i}{j}\right)^k
\le\frac{(j)_k}{k!j^k}
\le\frac{c_kj}{k!}.
\]
The final bound also holds for $j<k$, since the left-hand side is zero.

For each $k$-element subset $A\subseteq\{1,\ldots,n\}$, we have $\min_{i\in A}p_i=p_{\max A}=\sum_{j=\max A}^n(p_j-p_{j+1})$. Each such subset has $k!$ orders, so the preceding bounds imply
\[
\begin{aligned}
\Phi_P^{\mathrm{core}}(F,y)
&\le k!\sum_{\substack{A\subseteq\{1,\ldots,n\}\\|A|=k}}
\left(\min_{i\in A}p_i\right)\prod_{i\in A}p_i\\
&=k!\sum_{j=1}^n(p_j-p_{j+1})e_k(p_1,\ldots,p_j)\\
&\le c_k\sum_{j=1}^nj(p_j-p_{j+1})
=c_k\sum_{j=1}^np_j
\le c_k.
\end{aligned}
\]
This calculation also covers $n<k$ and zero weights.

For the reverse inequality, choose an integer $a\ge k$ attaining $c_k$, take $F=K_a$, and set $y_{V(F)}=1/a$ and all other weights to zero. Then $s(y)=1$ and $Y_\Gamma=1/a$ for every nonempty $\Gamma\subseteq V(F)$. The objective has $(a)_k$ summands, each equal to $a^{-(k+1)}$, and hence equals $c_k$. Thus $\Lambda_{\mathrm{core}}(P)=c_k$, and Corollary~\ref{cor:alon-first}, together with $|\Aut(P)|=2$, proves the odd-order estimate.
\end{proof}

For paths of odd order, Proposition~\ref{prop:odd-path-constant} exhibits a finite core attaining the supremum. For paths of even order, the empty core attains $\Lambda_{\mathrm{core}}(P_{2k})=2^k$. In general, by Lemma~\ref{lem:fixed-core-compactness}, the maximum over the weights on each fixed core is attained. By Proposition~\ref{prop:clique-core-reduction}, the core may be taken to be complete. However, neither result bounds the order of a maximizing core.

\begin{problem}\label{prob:finite-core-order}
For every nonempty finite simple graph $H$ without isolated vertices, do there exist an integer $t\ge0$ and a feasible vector $y$ such that $\Phi_H^{\mathrm{core}}(K_t,y)=\Lambda_{\mathrm{core}}(H)$?
\end{problem}

For the hypergraph variational formula, the order of a template needed for a given relative accuracy provides a quantitative version of the attainment question. Fix a nonempty $r$-graph $H$ without isolated vertices and put $h=|V(H)|$. For $0<\varepsilon<1$, define
\[
\kappa_H(\varepsilon)=\min\bigl\{|V(T)|:T\text{ is a finite simple $r$-graph and }c_H(T)\ge(1-\varepsilon)\Lambda_r(H)\bigr\}.
\]
The defining set is nonempty because $0<\Lambda_r(H)=\sup_Tc_H(T)$, so $\kappa_H(\varepsilon)$ is finite. Taking $K=\lceil h(h-1)/\varepsilon\rceil$ and $\eta=\varepsilon/(2K)$ in the proof of Proposition~\ref{prop:bounded-template-replacement}, we obtain complete blow-ups of templates with at most $L=K(\lfloor2K/\varepsilon\rfloor+1)=O_H(\varepsilon^{-3})$ vertices. Since the blow-ups constructed in this way have nonempty vertex classes, for all sufficiently large $m$ we have
\[
(1-\varepsilon)M_r(m,H)\le\max_{|V(T)|\le L}A_{H,T}(m).
\]
There are only finitely many template isomorphism types of order at most $L$. We divide by $m^{\rho^*(H)}$ and pass to the limit. By Proposition~\ref{prop:fixed-template-limit} and Theorem~\ref{thm:hypergraph-main}, we obtain $(1-\varepsilon)\Lambda_r(H)\le\max_{|V(T)|\le L}c_H(T)$. Consequently $\kappa_H(\varepsilon)\le L=O_H(\varepsilon^{-3})$ as $\varepsilon\downarrow0$.

The function $\kappa_H(\varepsilon)$ remains bounded as $\varepsilon\downarrow0$ if and only if some finite template $T$ satisfies $c_H(T)=\Lambda_r(H)$. One direction follows by using such a template for every $\varepsilon$. Conversely, a uniform order bound leaves finitely many template isomorphism types, and the maximum of their constants is at least $(1-\varepsilon)\Lambda_r(H)$ for arbitrarily small $\varepsilon$, so it equals $\Lambda_r(H)$.

\begin{problem}\label{prob:template-order}
Characterize the nonempty finite simple $r$-graphs $H$ without isolated vertices for which $\kappa_H(\varepsilon)$ remains bounded as $\varepsilon\downarrow0$. When it is unbounded, determine its asymptotic order.
\end{problem}

The projection bound $\Lambda_r(H)\le(r!)^{\rho^*(H)}$ is attained whenever $\rho^*(H)=h/r$. Indeed, let $q$ be the largest integer satisfying $\binom qr\le m$, and let $K_q^{(r)}$ be the complete $r$-graph on $q$ vertices. Since $(q-r+1)^r/r!\le\binom qr\le m<\binom{q+1}r\le(q+1)^r/r!$, we have $q=(r!m)^{1/r}+O_r(1)$. Adjoining pairwise vertex-disjoint $r$-edges on new vertices to $K_q^{(r)}$ completes the edge count to $m$, while the upper estimate follows from Proposition~\ref{prop:global-projection-bound}. Hence, for all sufficiently large $m$,
\[
(q)_h\le M_r(m,H)\le(r!m)^{h/r},
\qquad
M_r(m,H)=(r!)^{h/r}m^{h/r}+O_{H,r}(m^{(h-1)/r}).
\]
Theorem~\ref{thm:hypergraph-main} therefore identifies $\Lambda_r(H)=(r!)^{h/r}$. 


By the preceding estimate, the error is smaller than the main term by a factor of $m^{1/r}$ when $\rho^*(H)=h/r$. For general $H$, Theorem~\ref{thm:hypergraph-main} only asserts an error $o(m^{\rho^*(H)})$.

\begin{problem}\label{prob:rate-of-convergence}
For every nonempty finite simple $r$-graph $H$ without isolated vertices, do there exist constants $\delta=\delta(H)>0$ and $C=C(H)>0$ such that, for all sufficiently large integers $m$,
\[
\left|M_r(m,H)-\Lambda_r(H)m^{\rho^*(H)}\right|\le C m^{\rho^*(H)-\delta}?
\]
\end{problem}

Even for graphs with $\rho^*(H)=|V(H)|/2$, the value $\delta=1/2$ cannot always be improved. For an integer $q\ge3$ and $m=\binom q2$, we apply the shadow theorems of Kruskal~\cite{Kruskal1963} and Katona~\cite{Katona1968} to obtain $M_2(m,K_3)=(q)_3$. Since $\Lambda_2(K_3)=2^{3/2}$, the difference between the main term and the extremal count is
\[
(2m)^{3/2}-(q)_3=\frac32q^2+O(q)=\Theta(m)
\]
as $q\to\infty$. Thus no estimate in Problem~\ref{prob:rate-of-convergence} can hold for the triangle with $\delta>1/2$.

\begingroup\spaceskip=\fontdimen2\font plus \fontdimen3\font minus 1.2\fontdimen4
\section*{Acknowledgment}
Peiru Kuang is supported by Shanghai
Institute for Mathematics and Interdisciplinary Sciences, SIMIS-ID-26-AMS-003. Shuang Sun and Yan Wang are supported by the National Key R\&D Program of China under Grant No.~2022YFA1006400 and the National Natural Science Foundation of China under Grant No.~12571376. Jiasheng Zeng acknowledges support from the Research Grants Council of Hong Kong through the Hong Kong PhD Fellowship Scheme.

\par\endgroup

\section*{Declaration of the Use of Generative AI Tools}
In this work, ChatGPT assisted with mathematical derivations and discussions. For the graph case, the authors proposed the main idea and wrote a draft proof, which GPT-5.5 helped verify and refine. GPT-6 Astra subsequently extended the argument to globally rainbow embeddings in vertex-colored hypergraphs  and thereby obtained a proof for hypergraphs. The authors then verified and rewrote this proof. The proof of Conjecture~\ref{conj:alon-second} was developed autonomously by GPT-5.6 and was subsequently verified and rewritten by the authors. All mathematical results, arguments and proofs were independently reviewed and verified by the authors, who take full responsibility for the accuracy and content of this work.

\bibliographystyle{plain}
\bibliography{1104refs}
\end{document}